\documentclass[12pt,a4paper]{article}

\usepackage{tikz}

\usepackage{amssymb}
\usepackage{amsthm}
\usepackage{amsmath}
\usepackage{latexsym}
\usepackage{amsfonts}
\usepackage{color}
\usepackage{stmaryrd}
\usepackage{extarrows}
\usepackage{cite}
\usepackage{bm}
\usetikzlibrary{calc}
\usepackage{bbm}
\usepackage[left=3 cm, right=2 cm, top=2cm, bottom=2cm]{geometry}

\DeclareSymbolFont{rsfscript}{OMS}{rsfs}{m}{n}
\DeclareSymbolFontAlphabet{\mathrsfs}{rsfscript}

\DeclareMathOperator{\id}{\mathbbm{1}}

\DeclareMathOperator{\m}{\mbox{-}}
\DeclareMathOperator{\lcm}{lcm}

\DeclareMathOperator{\Sch}{Sch}
\DeclareMathOperator{\St}{Stab}
\DeclareSymbolFont{rsfscript}{OMS}{rsfs}{m}{n}
\DeclareSymbolFont{bbold}{U}{bbold}{m}{n}
\DeclareSymbolFontAlphabet{\mathbbold}{bbold}

\newtheorem{theorem}{Theorem}
\newtheorem{proposition}{Proposition}
\newtheorem{definition}{Definition}
\newtheorem{lemma}{Lemma}
\newtheorem{corollary}{Corollary}

\newtheorem{prob}{Problem}

\newcommand{\rf}{\rightarrow}

\newcommand{\ul}{\underline}
\newcommand{\wt}{\widetilde}

\newcommand{\inv}{^{-1}}

\newcommand{\oo}{\overline}

\def\vlongrightarrow{\relbar\joinrel\longrightarrow}
\def\vvlongrightarrow{\relbar\joinrel\vlongrightarrow}
\def\vvvlongrightarrow{\relbar\joinrel\vvlongrightarrow}
\def\vvvvlongrightarrow{\relbar\joinrel\vvvlongrightarrow}

\def\vlongmapright#1{\smash{\mathop{\vvlongrightarrow}\limits^{#1}}}

\newcommand{\longfr}[2]{\smash{\stackrel{\text{\tiny{$#1|#2$}}}{\vlongrightarrow}}}
\newcommand{\vlongfr}[2]{\smash{\stackrel{\text{\tiny{$#1|#2$}\,}}{\vvlongrightarrow}}}
\newcommand{\vvlongfr}[2]{\smash{\stackrel{\text{\tiny{$#1|#2$}\,}}{\vvvlongrightarrow}}}
\newcommand{\vvvlongfr}[2]{\smash{\stackrel{\text{\tiny{$#1|#2$}\,}}{\vvvvlongrightarrow}}}

\newcommand{\mapright}[1]{\smash{\stackrel{\text{\tiny{$#1$}}}{\vlongrightarrow}}}
\newcommand{\Pth}[1]{\smash{\stackrel{\text{\tiny{$#1$}}}{\longrightarrow}}}

\title{Freeness of automata groups vs boundary dynamics}
\author{Daniele D'Angeli\\
        Institut f\"{u}r Mathematische Strukturtheorie (Math C)\\
Technische Universit\"{a}t Graz\\
Steyrergasse 30, 8010 Graz, Austria.\\
        \texttt {dangeli@math.tugraz.at}
        \and
       	Emanuele Rodaro \\
        Centro de Matem\'atica, University of Porto\\
        Rua do Campo Alegre, 687, Porto, 4169-007, Portugal.\\				
        \texttt{emanuele.rodaro@fc.up.pt}
        }

\date{\today}

\begin{document}
\maketitle

\begin{abstract}
We prove that the boundary dynamics of the (semi)group generated by the enriched dual transducer characterizes the algebraic property of being free for an automaton group. We specialize this result to the class of bireversible transducers and we show that the property of being not free is equivalent to have a finite Schreier graph in the boundary of the enriched dual pointed on some essentially non-trivial point. From these results we derive some consequences from the dynamical, algorithmic and algebraic point of view. In the last part of the paper we address the problem of finding examples of non-bireversible transducers defining free groups, we show examples of transducers with sink accessible from every state which generate free groups, and, in general, we link this problem to the nonexistence of certain words with interesting combinatorial and geometrical properties.
\end{abstract}

\section{Introduction}
In 1980 R. I. Grigorchuk described the first example of a group of intermediate (i.e. faster
than polynomial and slower than exponential) growth. It later appeared that the most
natural way to study this group is by its action on the rooted binary tree, or to look at its generating automaton (Mealy machine). Over the last decades a new exciting direction of research focusing on finitely generated automata groups acting by automorphisms on rooted trees has been developed. It has proven to have deep connections with the theory of profinite groups and with complex
dynamics. In particular, many groups of this type satisfy a property of self-similarity, reflected on fractalness of some limit objects associated with them \cite{BaGriNe03, BDN, Nekra05}. In the spirit of the modern theory of dynamical systems, one is interested in the action of an automaton group (which is a countable group) on the uncountable set of right infinite sequences somehow endowed with the uniform measure. This action on the boundary seems to be very rich in this context and it is described by the structure of the corresponding Schreier graphs. There are examples of essentially free actions as well as examples of totally non-free actions, but so far no examples of actions that are neither essentially free nor totally non-free on the boundary. Therefore a series of open questions naturally appear in this setting \cite{basilica, GriSa13, Vershik}. One of the most intriguing question is what kind of groups can be generated by these Mealy machines. Among these an interesting family is constituted by finitely presented groups such as free groups and free products of finite groups. Examples of free groups generated by invertible Mealy machines are rare and difficult to obtain, see for instance \cite{GlaMoz05} where the authors claim that: ``Somewhat surprisingly, for example, it is not so easy to generate a free group''. Examples of free groups or free product of groups were found by Vorobets and Vorobets in \cite{VoVo2007,VoVo2010}, by Muntyan and Savchuk in the case of Bellaterra group \cite{Nekra05}, by Glasner and Mozes \cite{GlaMoz05}. In this paper we show a connection between the property of being free for such automata groups with their dynamics on the boundary. Indeed, in Section \ref{sec: freeness vs dynamics} we characterize the freeness of the groups generated by an invertible transducer in term of the dynamics on the boundary of group (semigroup) generated by its enriched dual. We first explore the general case, then we specialize it to the class of bireversible transducers. In the general case we show that being free is equivalent to the absence of finite Schreier (or orbital in the non-reversible case) graphs rooted at periodic points for the group (semigroup) generated by the enriched dual automaton. Subsequently, we consider bireversible transducers; we first show that for this class, enriching a transducer with a structure of inverse transducer does not change the generated group, however, its dynamics is richer and it strongly influences the algebraic structure of the group generated by its dual. In this framework, we spot a class of points in the boundary $(Q\cup Q^{-1})^{\omega}$, called essentially non-trivial, which ``represent'' elements on the Gromov boundary $\widehat{F_{Q}}$ of the free group $F_{Q}$. We show that, independently of the generated group, the Schreier graphs rooted at points in the complement of essentially non-trivial points are always finite. However, the freeness of the group generated by the dual is equivalent to the infiniteness of all the Schreier graphs pointed on essentially non-trivial elements. These characterizations have a series of consequences. For instance, we give an elementary proof (without passing throught Zelmanov's result) that if $G$ is a Burnside automaton group with bounded exponent, then $G$ is finite. Furthermore, we find sufficient conditions on an invertible transducer not to have any elements in the generated semigroup that act like the identity, or equivalently it does not posses ``positive relations'', i.e. relations of the form $u=1$. These semigroups seem to play also an important role in the open problem of looking for non-virtually nilpotent automata groups with all trivial boundary stabilizers (see, \cite{GriSa13}). Even if bireversible transducers generate groups with essentially free actions on the boundary, in all known examples they always exhibit a critical point with a non-trivial stabilizer. In this paper we give some necessary conditions for groups generated by reversible transducers. Indeed, we prove that if a reversible transducers defines a group having all trivial stabilizers in the boundary, then the semigroup of the dual have no ``positive relations'', in particular it is torsion-free, this extends a result in \cite{DaRo14} for the class of reversible invertible transducers. Furthermore, the group generated by its dual transducer does not possess this property of trivial stabilizers in the boundary. In this direction we restrict the scope of the search, indeed we show that the transducers defining these groups may lie both in the class of bireversible transducers, or in the complement of the class of reversible transducers. Moreover, if the first case occurs, then the corresponding group is Burnside. At the end of Section \ref{sec: freeness vs dynamics} we use the previously obtained characterizations to get some consequences from the algorithmic point of view. For instance, the general characterization leads to the equivalence of the problem of determining whether an automaton group is not free with the problem of detecting a finite orbital graph of a (reduced) periodic boundary-point.
\\
In the last section of the paper, we tackle the problem of finding free groups generated by automata with a sink state. Indeed, one common feature of all the aforementioned examples of free automata groups, is the fact that they are all defined by transducers which are bireversible. This leads to ask whether it is possible to generate a free group by means of transducers with a sink state, this case represents, in some sense, the opposite of bireversible transducers. We answer positively to this question by showing how to build an infinite series of Mealy machines with a sink state reachable from every state defining free groups. We show that this reachability condition is equivalent to the existence of $g$-regular elements on the boundary for every element $g$ in the generated group. In this case the resulting free groups do not act transitively on the corresponding tree, so it remains open the question of finding a free group generated by a non-bireversible automaton and acting transitively on the set of words of the same length in the alphabet. In this framework, we propose a combinatorial approach to deal with non-free groups defined by transducers whose sink state is accessible from every state, using certain words which we call fragile. Roughly speaking these are words representing minimal relations such that a small transformation bring them to reduced words. The last part of the paper is devoted to some examples of groups obtained by suitable colorings of the Cayley machines. In particular, we prove that the relations for such groups can be detected by using purely combinatorial properties of the dual automaton. Such observation enables us to study a special class of such machines generating free semigroups, and therefore groups with exponential growth.

\section{Preliminaries}\label{sec: preliminaries}
In the rest of the paper $A$ will denote a finite set, called \emph{alphabet}. A word $w$ over $A$ is a tuple $w=(w_{1},\ldots, w_{n})$ of element of $A$ which is more often represented as a string $w=w_{1}\ldots w_{n}$, and for convenience we will use both the notations freely. The set $A^{+}$ ($A^{*}$) of all finite non-empty words (words) over $A$ has a structure of free semigroup (monoid) on $A$ with respect to the usual operation of concatenation of words (and with identity the empty word $1$). By $A^{\le n}$, $A^{\ge n}$ and $A^{n}$ we denote the set of words of length less then, greater, equal to $n$, respectively. By $A^{\omega}$ we denote the set of words on $A$ infinite to the right, we use the vector notation, and for an element $\ul{u}=u_{1}u_{2}\ldots u_{i}\ldots\in A^{\omega}$ the prefix of length $k>0$ is denoted by $\ul{u}[k]=u_{1}u_{2}\ldots u_{k}$, while the factor $u_{i}\ldots u_{j}$ is denoted by $\ul{u}[i,j]$. By $\tilde{A}=A\cup A\inv$ we denote the \emph{involutive alphabet} where $A\inv$ is the set of \emph{formal} inverses of elements $A$. The operator $\inv\colon A\rightarrow A\inv$ sending $ a\mapsto a\inv$ is extended to an involution on the free monoid $\wt{A}^*$ through
$$
1\inv = 1, \;\; (a\inv)\inv=a, \;\; (uv)\inv=v\inv u\inv\;\;\; (a\in
A;\;u,v\in\wt{A}^*).
$$
Let $\sim$ be the congruence on $\wt{A}^*$ generated by the relation
$\{(aa\inv,1)\mid a \in \wt{A} \}$. The quotient $F_A= \wt{A}^*/\sim$ is the \emph{free group} on $A$, and let $\sigma\colon \wt{A}^* \to F_A$ be the canonical homomorphism. The set of all reduced words on $\wt{A}^*$ may be compactly written as
$$
R_A=\tilde{A}^*\setminus\bigcup_{a\in\tilde{A}}\tilde{A}^*aa\inv\tilde{A}^*
$$
For each $u\in\wt{A}^*$, we denote by $\overline{u}\in R_A$ the (unique) reduced word $\sim$-equivalent to $u$. With a slight abuse in the notation we often identify the elements of $F_A$ with their reduced representatives, i.e. $\sigma(u)=\overline{u}$; this clearly extends to subsets $\sigma(L)=\oo{L}$, $L\subseteq \wt{A}^{*}$.
\\
In this paper we strongly follow the geometric/language theoretic approach to automata groups developed in \cite{DaRo14}, for this reason we treat transducers and automata as labelled digraphs. Using Serre's approach, an $A$-digraph is a tuple $\Gamma=(V,E,A,\iota,\tau,\mu)$, where $V$ is the set of vertices (or states), $E$ is the set of edges (or transitions), $\iota, \tau$ are functions from $E$ into $V$ giving the initial and terminal vertices, respectively, and $\mu:E\rightarrow A$ is the \emph{labeling} map. It is more compact to depict an edge $e\in E$ with $q=\iota(e), q'=\tau(e)$, $\mu(e)=a$ as $e=(q\mapright{a}q')$. A \emph{path} is an ordered sequence of edges $p=e_{1}\ldots e_{k}$ such that $\tau(e_{i})=\iota(e_{i+1})$ for $i=1,\ldots, k-1$, and we say that the origin of $p$ is $\iota(p)=\iota(e_{1})$ and the terminal vertex is $\tau(p)=\tau(e_{k})$. The \emph{label} of the path $p$ is the word $\mu(p)=\mu(e_{1})\ldots \mu(e_{k})$, and we graphically represent this path as $p=(v\vlongmapright{\mu(p)}v')$. When we fix a vertex $v\in V$ (a \emph{base point}), the pair $(\Gamma, v)$ can be seen as a language recognizer ($A$-automaton), whose language recognized is the set:
$L(\Gamma,v)=\{\mu(p):p\mbox{ is a path in }\Gamma\mbox{ with }\iota(p)=\tau(p)=v\}$.
When we pinpoint the vertex $v$, we implicitly assume that the underlying $A$-digraph of $(\Gamma,v)$ is the connected component of $\Gamma$ containing $v$, and $\|(\Gamma, v)\|$ denotes the cardinality of the set of vertices of $(\Gamma,v)$. A morphism of $A$-digraphs ($A$-automata) are defined as usual. The $A$-labelled graph $\Gamma$ is called \emph{complete} (\emph{deterministic}) if for each vertex $v\in V$ and $a\in A$ there is (at most) an edge $e\in E$ with $\iota(e)=v$ and $\mu(e)=a$. One usually refers to a deterministic and complete $A$-labelled graph $\Gamma$ with a finite number of vertices as \emph{semiautomaton} \cite{HowieAuto} and it can be equivalently described by a $3$-tuple $\mathcal{A}=(Q,A,\delta)$ where $Q$ is a finite set of states, $A$ is a finite alphabet, $\delta:Q\times A\rf Q$ is the \emph{transition function}. Fixing a base-point $q\in Q$ the language recognized by the automaton (DFA) $(\mathcal{A}, q)$ is the set $L(\mathcal{A},q)=\{u\in A^{*}: \delta(q,u)=q\}$. Note that this notation is compatible with the one presented above. The map $\delta$ induces an action $Q\overset{\cdot}{\curvearrowleft} A^{*}$ of $A^{*}$ on $Q$ defined inductively by the formula $q\cdot (a_{1}\ldots a_{n})=\delta(q, a_{1})\cdot  (a_{2}\ldots a_{n}), \forall q\in Q, \: \forall a_{1}\ldots a_{n}\in A^{*}$. The semiautomaton $\mathcal{A}$ is called \emph{reversible} whenever this action is a permutation. In this paper we consider alphabetical transducers with the same input and output alphabet, for further details on the general theory of automata and transducers we refer the reader to \cite{Eil, HowieAuto}. A \emph{finite state Mealy automaton}, shortly a transducer (or a machine), is a $4$-tuple $\mathrsfs{A} = (Q,A,\delta,\lambda)$ where $(Q,A,\delta)$ is a semiautomaton, while $\lambda:Q\times A\rf A$ is called the \emph{output function}. This function defines an action $Q\overset{\circ}{\curvearrowright} A^{*}$ of $Q$ on $A^{*}$ defined inductively by
$$
q\circ (a_{1}\ldots a_{n})=\lambda(q,a_{1})\left((q\cdot a_{1})\circ (a_{2}\ldots a_{n})\right)
$$
Both $Q\overset{\circ}{\curvearrowright} A^{*}$ and $Q\overset{\cdot}{\curvearrowleft} A^{*}$ can be naturally extended to $Q^{*}$; we refer to the pair $(Q^{*}\overset{\cdot}{\curvearrowleft} A^{*}$, $Q^{*}\overset{\circ}{\curvearrowright} A^{*})$ as the associated \emph{coupled-actions} of the transducer $\mathrsfs{A}$, and henceforth we will use the notation $\mathrsfs{A}=(Q,A,\cdot,\circ)$ underlying these actions.
\\
From the geometrical point the transducer $\mathrsfs{A}$ can be visualized as an $A\times A$-labelled digraph with edges of the form $q\mapright{a|b}q'$ whenever $q\cdot a=q'$ and $q\circ a=b$, and we will make no distinction between the transducer and the digraph notation. Considering just the input or the output labeling, we may define the \emph{input automaton} $\mathrsfs{A}_{\mathcal{I}}$ and the \emph{output automaton} $\mathrsfs{A}_{\mathcal{O}}$ having edges $q\mapright{a}q'$, $q\mapright{b}q'$ whenever $q\mapright{a|b}q'$ is an edge in $\mathrsfs{A}$, respectively. Another operation is the \emph{product} of two machines $\mathrsfs{A} = (Q,E,A\times A,\iota, \tau, \mu)$, $\mathrsfs{B} = (T,D,A\times A,\iota', \tau', \mu')$, this is the transducer $\mathrsfs{A}\mathrsfs{B} = (Q\times T,F,A\times A,\oo{\iota},\oo{\tau},\oo{\mu})$ whose edges are given by $(q,q')\mapright{a|b}(p,p')$ whenever $q\mapright{a|c}p$ is an edge in $E$ and  $q'\mapright{c|b}p'$ is an edge in $D$. The $k$-th power of the machine $\mathrsfs{A}$ is defined inductively by $\mathrsfs{A}^{k}=(\mathrsfs{A}^{k-1})\mathrsfs{A}$, and we put $\mathrsfs{A}^{k}_{\mathcal{I}}=(\mathrsfs{A}^{k})_{\mathcal{I}}$.
From the algebraic point of view the action $Q^{*}\overset{\circ}{\curvearrowright} A^{*}$ gives rise to a semigroup $\mathcal{S}(\mathrsfs{A})$ generated by the endomorphisms $\mathrsfs{A}_{q}$, $q\in Q$, of the rooted tree identified with $A^{*}$ defined by $\mathrsfs{A}_{q}(u)=q\circ u$, $u\in A^{*}$. For $q_{1},\ldots, q_{m}\in Q$ we may use the shorter notation $\mathrsfs{A}_{q_{1}\ldots q_{m}}=\mathrsfs{A}_{q_{1}}\ldots \mathrsfs{A}_{q_{m}}$. An important role in group theory is played by groups defined by invertible transducers, for more details we refer the reader to \cite{Nekra05}. A transducer $\mathrsfs{A} = (Q,A,\cdot,\circ)$ is called \emph{invertible} whenever the map $\lambda(q,\circ):A\rf A$ is a permutation. In this case all the maps $\mathrsfs{A}_{q}$, $q\in Q$, are automorphisms of the rooted regular tree identified with $A^{*}$, and the group generated by these automorphisms is denoted by $\mathcal{G}(\mathrsfs{A})$ (with identity $\id$). Henceforth a generator $\mathrsfs{A}_{q}$ of $\mathcal{G}(\mathrsfs{A})$ (or $\mathcal{S}(\mathrsfs{A})$) is identified with the element $q\in Q$, and its inverse with the formal inverse $q^{-1}\in Q^{-1}=\{q^{-1}:q\in Q\}$. Note that the actions of the inverses $Q^{-1}$ are given by the inverse (transducer) $\mathrsfs{A}^{-1}$ having $Q^{-1}$ as the set of vertices, and there is an edge $q^{-1}\mapright{a|b}p^{-1}$ in $\mathrsfs{A}^{-1}$ whenever $q\mapright{b|a}p$ is an edge in $\mathrsfs{A}$. The action of $\mathcal{G}(\mathrsfs{A})$ on $A^{*}$, in case $\mathrsfs{A}$ is invertible (or $\mathcal{S}(\mathrsfs{A})$ in the more general case), can be naturally extended on the \textit{boundary} $A^{\omega}$ of the tree. This action gives rise to the so called orbital graph. In general, given a finitely generated semigroup $S$ with set of generators $A$ which acts on the left on a set $X$, if $\pi:A^{*}\rightarrow S$ denotes the canonical map, then the \emph{orbital graph} $\Gamma(S,A,X)$ is defined as the $A$-digraph with set of vertices $X$, and there is an edge $x\mapright{a}y$ whenever $\pi(a)x=y$. When we want to pinpoint the connected component containing the element $y\in X$ we use the shorter notation $\Gamma(S,A,X,y)$ instead of $\left (\Gamma(S,A,X),y\right)$. Note that in the realm of groups, this notion correspond to the notion of Schreier graph. In particular for the group $\mathcal{G}(\mathrsfs{A})$ and $v\in A^{\ast}\sqcup A^{\omega}$, if $\St_{\mathcal{G}(\mathrsfs{A})}(v)=\{g\in \mathcal{G}(\mathrsfs{A}) \ : \ g(v)=v\}$ denotes the stabilizer of $v$, the Schreier graph $\Sch(\St_{\mathcal{G}(\mathrsfs{A})}(v), Q\cup Q^{-1})$ corresponds to the connected component of the orbital graph $\Gamma(\mathcal{G}(\mathrsfs{A}),Q\cup Q^{-1},A^{*}\sqcup A^{\omega})$ pinpointed by $v$. Equivalently, from the automaton point of view, we have
$$
\left(\Sch(\St_{\mathcal{G}(\mathrsfs{A})}(v), Q\cup Q^{-1}),\St_{\mathcal{G}(\mathrsfs{A})}(v)\right)\simeq\Gamma(\mathcal{G}(\mathrsfs{A}),Q\cup Q^{-1},A^{*}\sqcup A^{\omega},v)
$$
Another two important classes of transducers that we consider throughout the paper are the \emph{reversible} and \emph{bireversible} machines. A transducer $\mathrsfs{A}$ is called reversible whenever $\mathrsfs{A}_{\mathcal{I}}$ is a reversible semiautomaton, and it is called bireversible if in addition also $\mathrsfs{A}_{\mathcal{O}}$ is a reversible semiatomaton, hence in this case $\mathrsfs{A}$ must be necessarily invertible.

\subsection{A geometric perspective via the enriched dual}
The approach used in this paper strongly follows the ideas developed in \cite{DaRo14} using enriched duals. This is a compact way to deal with the relations of $\mathcal{G}(\mathrsfs{A})$, we briefly recall this notion and some of its consequences. An $\wt{A}$-digraph $\Gamma$ is called \emph{involutive} if whenever $p\mapright{a}q$ is and edge of $\Gamma$, so is $q\mapright{a^{-1}}p$. An involutive graph $\Gamma$ is called \emph{inverse} if in addition $\Gamma$ is deterministic. In depicting an inverse graph we can just draw one out of the two edges $p\mapright{a}q$, $q\mapright{a^{-1}}p$, $a\in\wt{ A}$; this corresponds to the choice of an orientation $E^{+}$ on the set of edges $E$; if in choosing this orientation and forgetting the orientation of the edges we obtain a tree, then we call it an \emph{inverse tree}. When a base-point $v\in V(\Gamma)$ is fixed, the pair $(\Gamma,v)$ is often referred as an \emph{inverse $A$-automaton} (or simply inverse automaton when the alphabet $A$ is clear from the context); and in case $\Gamma$ is an inverse tree, the pair $(\Gamma,v)$ is a called a \emph{rooted inverse tree}. For inverse automata there is an important property (belonging to the folklore) which relates languages to morphisms:
\begin{proposition}\label{prop: immersion lang hom}
Let $\Gamma_{1},\Gamma_{2}$ be two inverse graphs, and $q_{1},q_{2}$ be two vertices belonging to $\Gamma_{1},\Gamma_{2}$, respectively. Then $L(\Gamma_{1},v_{1})\subseteq L(\Gamma_{2},v_{2})$ if and only if there is a morphism $\varphi\colon (\Gamma_{1},v_{1})\rightarrow (\Gamma_{2},v_{2})$. Furthermore, $(\Gamma_{1},v_{1})$ is the minimal inverse automaton (up to isomorphism) recognizing $L(\Gamma_{1},v_{1})$.
\end{proposition}
Another important feature of inverse automata is their connection with subgroups of $F_{A}$ via the notion of Stallings automaton $\mathcal{S}_{t}(H)$ of a finitely generated group $H\le F_{A}$, we refer the reader to \cite{BarSil,KaMi02} for further details. While inverse automata on the alphabet $A$ essentially represent subgroups of $F_{A}$, there is an important tool recently introduced by Silva in \cite{SilvaVirtual} which is a compact way to deal with rational maps on $F_{A}$: the class of \emph{inverse transducers}. For our purposes, an inverse transducer is an inverse $(\wt{A}\times \wt{A})$-digraphs where the involution is given by $(a,b)\mapsto (a^{-1},b^{-1})$. A reversible machine can be always ``enriched'' with a structure of inverse transducer:
\begin{lemma}\cite[Lemma 1]{DaRo14}
Let $\mathrsfs{A}=(Q,A,\cdot,\circ)$ be a reversible (bireversible) transducer, then it can be extended to a reversible (bireversible) inverse transducer $\mathrsfs{A}^{-}=(Q,\wt{A},\cdot,\circ)$ obtained from $\mathrsfs{A}$ by adding to each edge $q\mapright{a|b}p$ of $\mathrsfs{A}$ the edge $p\vlongfr{a^{-1}}{b^{-1}}q$.
\end{lemma}
Note that the property of being invertible is not preserved on the alphabet $\wt{A}$ in the passage from $\mathrsfs{A}$ to its associated inverse transducer $\mathrsfs{A}^{-}$. However, if $\mathrsfs{A}$ is bireversible this property is preserved. This operator also commutes with the product:
\begin{proposition}\cite[Proposition 4]{DaRo14}\label{prop: inverse of products}
Let $\mathrsfs{A}$, $\mathrsfs{B}$ be two reversible transducers, then:
$$
\left(\mathrsfs{A}\mathrsfs{B}\right)^{-}=\mathrsfs{A}^{-}\mathrsfs{B}^{-}
$$
In particular, there is a epimorphism from $\mathcal{S}(\mathrsfs{A}^{-})$ onto $\mathcal{S}(\mathrsfs{A})$.
\end{proposition}
The \emph{dual} of of a transducer is a well known concept in automata groups, and it is a useful tool already used in many papers, see for instance \cite{Pic12,StVoVo2011, VoVo2007,VoVo2010}. Formally, given a transducer $\mathrsfs{A} = (Q,A,\cdot,\circ)$ this is the (well defined) transducer $\partial\mathrsfs{A} = (A,Q,\circ,\cdot)$ such that $p\longfr{a}{b}q$ is an edge of $\mathrsfs{A} $ if and only if $a\longfr{p}{q}b$ is an edge of $\partial\mathrsfs{A}$. If the pair $(Q^{*}\overset{\cdot}{\curvearrowleft} A^{*},Q^{*}\overset{\circ}{\curvearrowright} A^{*})$ are the couple-actions associated to $\mathrsfs{A}$, then $(A^{*}\overset{\circ}{\curvearrowleft} Q^{*},A^{*}\overset{\cdot}{\curvearrowright} Q^{*})$ are the coupled-actions associated to $\partial\mathrsfs{A}$. Since sometimes we work simultaneously with both $\mathrsfs{A}$ and $\partial\mathrsfs{A}$, we use both the coupled-actions $(Q^{*}\overset{\cdot}{\curvearrowleft} A^{*},Q^{*}\overset{\circ}{\curvearrowright} A^{*})$ and $(A^{*}\overset{\circ}{\curvearrowleft} Q^{*},A^{*}\overset{\cdot}{\curvearrowright} Q^{*})$, and the order in which the actions appear will implicitly discriminate which of the two automata $\mathrsfs{A}$ and $\partial\mathrsfs{A}$ we are working with. For instance, let $u\in Q^{*}, v\in A^{*}$, if we write $u\cdot v$ or $u\circ v$ we are implicitly considering $\mathrsfs{A}$, otherwise a formulas $v\cdot u$, $v\circ u$ mean that we are working with the coupled-actions of $\partial\mathrsfs{A}$. Note that with this convention
$$
u\cdot v=(v^{R}\cdot u^{R})^{R},\quad u\circ v=(v^{R}\circ u^{R})^{R}, \quad \forall u\in Q^{*}, v\in A^{*}
$$
hold, where the \emph{mirror} operator $\circ^{R}$ is defined by $(u_{1}\ldots u_{k})^{R}\mapsto (u_{k}\ldots u_{1})$.
\\
The following proposition sums up some relationships between a transducer and its dual.
\begin{proposition}\cite{VoVo2007,VoVo2010,DaRo14}\label{prop: dual prop}
Let $\mathrsfs{A}$ be a transducer, then:
\begin{enumerate}
 \item[i)] $\mathrsfs{A}$ is invertible if and only if $\partial\mathrsfs{A}$ is reversible;
 \item[ii)] $\mathrsfs{A}$ is a reversible invertible transducer if and only if $\partial\mathrsfs{A}$ is a reversible invertible transducer;
 \item[iii)] $\mathrsfs{A}$ is bireversible if and only if $\partial\mathrsfs{A}$ is bireversible.
\end{enumerate}
\end{proposition}
In the case $\mathrsfs{A}$ is invertible, we can extend the actions of $\mathrsfs{A}$ to the disjoint union $\mathrsfs{A}\sqcup \mathrsfs{A}^{-1}$ in the obvious way. This extension is clearly reflected on the coupled-action $(\wt{Q}^{*}\overset{\cdot}{\curvearrowleft} A^{*},\wt{Q}^{*}\overset{\circ}{\curvearrowright} A^{*})$ which is also equivalent to the action of the group $\mathcal{G}(\mathrsfs{A})$ on $A^{*}$. It is straightforward to check that $\partial\left(\mathrsfs{A}\sqcup \mathrsfs{A}^{-1}\right)=(\partial\mathrsfs{A})^{-}$ holds (see \cite[Lemma 2]{DaRo14}). The transducer $(\partial\mathrsfs{A})^{-}$ is called the \emph{enriched dual} of $\mathrsfs{A}$, and by the previous equation it is clear that $(\partial\mathrsfs{A})^{-}$ is a compact tool to geometrically encode algebraic and topological properties of the group $\mathcal{G}(\mathrsfs{A})$. For instance, the following theorem is a crucial ingredient used through the paper and it characterizes the relations defining the group $\mathcal{G}(\mathrsfs{A})$. This fact has been already used in many papers on the subject implicitly, here we give it in a formalized form.
\begin{theorem}\cite{DaRo14}\label{theo: charact relations}
Let $\mathrsfs{A}=(Q,A,\cdot,\circ)$ be an invertible transducer, with $\mathcal{G}(\mathrsfs{A})\simeq F_{Q}/N$. Consider the transducer $(\partial\mathrsfs{A})^{-}=(A,\wt{Q},\circ,\cdot)$, and let
\begin{equation}\label{eq: containing relation}
\mathcal{N}\subseteq \bigcap_{a\in A}L\left((\partial\mathrsfs{A})^{-},a\right)
\end{equation}
be the maximal subset invariant for the action $A\overset{\cdot}{\curvearrowright} \wt{Q}^{*}$.Then $N=\oo{\mathcal{N}}$.
\end{theorem}
We denote by $\mathcal{C}(G,A)$ the \emph{Cayley automaton} of the group $G$ with symmetric generating set $A\cup A^{-1}$ and base-point the identity of $G$. Note that if $G=F_{A}/H$, then $L(\mathcal{C}(G,A))=\sigma^{-1}(H)$. The following theorem gives a way to represent the Schreier automata from the powers of the enriched dual as well as a way to build (at the limit) the Cayley automaton of the group $\mathcal{G}(\mathrsfs{A})$.
\begin{theorem}\cite[Theorem 3]{DaRo14}\label{thm:schreier}
Let $G=\mathcal{G}(\mathrsfs{A})=F_{Q}/N$. For any $k\ge 1$ put $\mathcal{D}_{k}=((\partial\mathrsfs{A})^{-})^{k}_{\mathcal{I}}$. The following facts hold.
\begin{enumerate}
\item[i)] If $v=a_{1}\ldots a_{k}\in A^{k}$, and $H=\St_{G}(v)$. Then
$$
\left(\mathcal{D}_{k}, v\right)\simeq (\Sch(H,\wt{Q}),H)
$$
\item[ii)] If $\ul{v}=a_{1}a_{2}\ldots\in A^{\omega}$, and $H=\St_{G}(v)$, then $\varprojlim \{\mathcal{D}^{k}_{\mathcal{I}}\}_{k\ge 1}=\mathcal{D}_{\mathcal{I}}^{\infty}$ is an inverse graph such that
$$
(\mathcal{D}_{\mathcal{I}}^{\infty}, \ul{v})\simeq  (\Sch(H,\wt{Q}),H)
$$
\item[iii)] Let $(\mathcal{N}_{k},v_{k})=\prod_{v\in A^{k}} (\mathcal{D}_{k},v)$. Then $\mathcal{C}(G/\St_{G}(k),Q)\simeq (\mathcal{N}_{k},v_{k})$.
Furthermore, the inclusions $L(\mathcal{N}_{k},v_{k})\subseteq L(\mathcal{N}_{k-1},v_{k-1})$ induces maps $$\psi_{k,k-1}:(\mathcal{N}_{k},v_{k})\rf (\mathcal{N}_{k-1},v_{k-1})$$ giving rise to the inverse system $\left(\{(\mathcal{N}_{k},v_{k})\}_{k\ge 1}, \psi_{i,j}\right)$ such that if we put $\mathcal{N}^{\infty}=\varprojlim \{(\mathcal{N}_{k},v_{k})\}_{k\ge 1}$, then
$$
\mathcal{C}(G,Q)\simeq \mathcal{N}^{\infty}
$$
\end{enumerate}
\end{theorem}

\section{Freeness characterizations in term of the dynamics on the boundary}\label{sec: freeness vs dynamics}

The problem of determining whether an invertible transducer $\mathrsfs{A}$ generates a free group seems to be very interesting and it has not a complete solution. In this section we provide a characterization in term of the dynamics of its enriched dual transducer: we show that the property of not being free is equivalent to the existence of a finite orbital graph in the boundary of the enriched dual automaton pointed on some periodic point. Unfortunately, we are not able to give a characterization in term of just existence of a finite Schreier graph on the boundary. However, in the next subsection we show such a characterization for the class of bireversible transducers. We begin with the following lemma.

\begin{lemma}\label{lem: periodic in finite schr}
Let $\mathrsfs{A}=(Q,A,\cdot,\circ)$ be an invertible automaton, and let $S=\mathcal{S}(\partial \mathrsfs{A}^{-})$. The following hold:
\begin{itemize}
\item Let $y\in \wt{Q}^{*}$, if $\|\Gamma(S,A,\wt{Q}^{\omega},y^{\omega})\|<\infty$, then any vertex of $\Gamma(S,A,\wt{Q}^{\omega},y^{\omega})$ is a periodic point $z^{\omega}$ with $|z|=|y|k(z)$ for some integer $k(z)$ depending on $z$.
\item Let $x,y\in \wt{Q}^{*}$, if $\|\Gamma(S,A,\wt{Q}^{\omega},xy^{\omega})\|<\infty$, then any vertex of $\Gamma(S,A,\wt{Q}^{\omega},y^{\omega})$ is an almost-periodic point $hr^{\omega}$ with $|h|=|x|$ and $|r|=|y|k(r)$ for some integer $k(r)$ depending on $r$.
\end{itemize}
\end{lemma}
\begin{proof}
Since $\Gamma(S,A,\wt{Q}^{\omega},y^{\omega})$ is finite and $y^{\omega}$ is periodic, it is easy to see that the first statement holds if we show that if $x^{\omega}\mapright{a} t$ is an edge in $\Gamma(S,A,\wt{Q}^{\omega},y^{\omega})$ where $x\in\wt{Q}^{*}$ with $|x|=|y|k(x)$ and $t\in\wt{Q}^{\omega}$, then $t=z^{\omega}$ for some $z\in\wt{Q}^{*}$ with $|z|=|y|k(z)$. Since the transition monoid of $\partial \mathrsfs{A}^{-}$ is a finite group of some order $m$, we have that for any $u\in\wt{Q}^{*}$ and $a\in A$, $a\circ u^{m}=a$, from which it easily follows
$$
(a\cdot x^{m})^{\omega}=a\cdot (x^{m})^{\omega}=a\cdot x^{\omega}=t
$$
Taking $z=a\cdot x^{m}$ we get $t=z^{\omega}$ and $|z|=|y|(k(z)m)$, from which the previous statement follows.
\\
The second case is analogous. Indeed, it holds if we show that in case $x'z^{\omega}\mapright{a} t$ is an edge in $\Gamma(S,A,\wt{Q}^{\omega},xy^{\omega})$ where $x',z\in\wt{Q}^{*}$ with $|x|=|x'|$ and $|z|=|y|k(z)$ and $t\in\wt{Q}^{\omega}$, then $t=hr^{\omega}$ for some $h\in\wt{Q}^{*}$ with $|h|=|x|$ and $|r|=|y|k(r)$. Similarly as above we have
$$
a\cdot (x'(z^{m})^{\omega})=h (a\cdot z^{m})^{\omega}=t
$$
with $h=a\cdot x'$. The statement follows taking $r=a\cdot z^{m}$, from which we get $t=hr^{\omega}$ with $|h|=|x'|$ and $|r|=|y|(k(z)m)$.
\end{proof}
We are now ready to prove the general characterization.
\begin{theorem}\label{theo: characterization free}
Let $\mathrsfs{A}=(Q,A,\cdot,\circ)$ be an invertible automaton, and let $G:=\mathcal{G}(\mathrsfs{A})$, $S=\mathcal{S}(\partial \mathrsfs{A}^{-})$. Then the following are equivalent:
\begin{itemize}
\item [i)] $G=F_{Q}/N$ is not free with a non-trivial relation $y^{r}\in\wt{Q}^{*}$ for some positive integer $r$;
\item[ii)] there is a periodic point $y^{\omega}$ with $y\in \widetilde{Q}^{\ast}$ and $\oo{y}\neq 1$ such that $$ \|\Gamma(S,A,\wt{Q}^{\omega},y^{\omega}) \|<\infty$$
\end{itemize}
\end{theorem}
\begin{proof}
$i)\Rightarrow ii)$. Let $y^{r}\in \wt{Q}^{*}$ be a relation of $G$ with $\oo{y}\neq 1$. By Theorem \ref{theo: charact relations} the set $Y=\{u\cdot y^{r}: u\in A^{\ast}\}$ is invariant for the action $A\overset{\cdot}{\curvearrowright} \wt{Q}^{*}$. Furthermore, $Y\subseteq \bigcap_{a\in A}L\left((\partial\mathrsfs{A})^{-},a\right)$, hence a simple computation shows that
$$
a\cdot z^{\omega}=(a\cdot z)^{\omega}, \quad \forall z\in Y, \, \forall a\in A
$$
holds. Thus, since $Y$ is finite we deduce $\|\Gamma(S,A,\wt{Q}^{\omega},y^{\omega}) \|<\infty$.
\\
$ii)\Rightarrow i)$. By Lemma \ref{lem: periodic in finite schr} the vertices of $\Gamma(S,A,\wt{Q}^{\omega},y^{\omega})$ are $y_{1}^{\omega},\ldots, y_{n}^{\omega}$ with $|y_{1}|=k_{1},\ldots, |y_{n}|=k_{n}$. Let $m$ be the order of the transition monoid of $\partial \mathrsfs{A}^{-}$, and let $r=\lcm\{m,k_{1},\ldots, k_{n}\}$, by definition of $\Gamma(S,A,\wt{Q}^{\omega},y^{\omega})$ we have that if $y_{i}^{\omega}\mapright{a}y_{j}^{\omega}$ is an edge in $\Gamma(S,A,\wt{Q}^{\omega},y^{\omega})$, then $a\cdot y_{i}^{r}=y_{j}^{r}$. Hence, since $m$ divides $r$, we have that the set
$$
Y=\{y_{1}^r,\ldots, y_{n}^{r}\}\subseteq \bigcap_{a\in A}L\left((\partial\mathrsfs{A})^{-},a\right)
$$
 is invariant for the action $A\overset{\cdot}{\curvearrowright} \wt{Q}^{*}$. Furthermore, since $\oo{y}\neq 1$ and $y^{r}\in Y$, we have that $y^{r}$ is a non-trivial relation of $G$, whence $G$ is not free.
\end{proof}

\subsection{The bireversible case}

Enriching a reversible transducer $\mathrsfs{A}=(Q,A,\cdot,\circ)$ with a structure of inverse transducer, in general, changes the structure of its generated semigroup: by Proposition \ref{prop: inverse of products} there is an epimorphism $\mathcal{S}(\mathrsfs{A}^{-})\rightarrow\mathcal{S}(\mathrsfs{A})$. In case $\mathrsfs{A}$ is bireversible, both $\mathrsfs{A}^{-}$ and $\mathrsfs{A}$ are invertible and thus they define a group; also in this case it is straightforward to check that there is an epimorphism $\mathcal{G}(\mathrsfs{A}^{-})\rightarrow\mathcal{G}(\mathrsfs{A})$. Answering to an open question raised in \cite{DaRo14}, quite surprisingly, in the next theorem we show that this is actually an isomorphism.
\begin{theorem}\label{theo: enriched is equal}
Let $\mathrsfs{A}=(Q,A,\cdot,\circ)$ be a bireversible transducer, then $\mathcal{G}(\mathrsfs{A}^{-})\simeq\mathcal{G}(\mathrsfs{A})$.
\end{theorem}
\begin{proof}
Note that both $\mathcal{G}(\mathrsfs{A}^{-}\sqcup (\mathrsfs{A}^{-})^{-1})\simeq \mathcal{G}(\mathrsfs{A}^{-})$ and $$\mathcal{G}\left((\partial\mathrsfs{A})^{-}\sqcup [(\partial\mathrsfs{A})^{-}]^{-1}\right)\simeq \mathcal{G}\left((\partial\mathrsfs{A})^{-}\right)$$ holds. Furthermore, by a simple computation we get:
\begin{equation}\label{eq: partial and inverses}
\partial\left(\mathrsfs{A}^{-}\sqcup (\mathrsfs{A}^{-})^{-1}\right)=[\partial \mathrsfs{A}^{-}]^{-}=\left[\partial\mathrsfs{A}\sqcup \partial\mathrsfs{A}^{-1}\right]^{-}=(\partial\mathrsfs{A})^{-}\sqcup [(\partial\mathrsfs{A})^{-}]^{-1}
\end{equation}
where in the last equality we have used $ [(\partial\mathrsfs{A})^{-1}]^{-}= [(\partial\mathrsfs{A})^{-}]^{-1}$. Let $\mathcal{G}(\mathrsfs{A})=F_{Q}/N$, by Theorem \ref{theo: charact relations} there is a maximal subset $\mathcal{N}\subseteq \bigcap_{a\in A}L\left((\partial\mathrsfs{A})^{-},a\right)$ invariant for the action $A\overset{\cdot}{\curvearrowright} \wt{Q}^{*}$ with $N=\oo{\mathcal{N}}$. Since there is an epimorphism $\mathcal{G}(\mathrsfs{A}^{-})\rightarrow\mathcal{G}(\mathrsfs{A})$, by Theorem \ref{theo: charact relations} and (\ref{eq: partial and inverses}) to prove the isomorphism it is enough to show the inclusion
\begin{equation}\label{eq: inclusion to prove}
\mathcal{N}\subseteq \bigcap_{a\in A^{-1}}L\left([(\partial\mathrsfs{A})^{-}]^{-1},a\right)
\end{equation}
and that $\mathcal{N}$ is invariant for the action $A^{-1}\overset{\cdot}{\curvearrowright} \wt{Q}^{*}$ in the transducer $[(\partial\mathrsfs{A})^{-}]^{-1}$. It is not difficult to check that
\begin{equation}\label{eq: charact language}
L\left([(\partial\mathrsfs{A})^{-}]^{-1},a^{-1}\right)=\left\{v: a\cdot v\in L\left((\partial\mathrsfs{A})^{-},a\right)\right\}
\end{equation}
holds. Fix a word $u\in\mathcal{N}$. Since $(\partial\mathrsfs{A})^{-}$ is invertible, for any $a\in A$ there is an integer $m(a)$ such that $a^{m(a)}\cdot u=u$. Moreover, since $\mathcal{N}$ is invariant for the action $A\overset{\cdot}{\curvearrowright} \wt{Q}^{*}$, we get
\begin{equation}\label{eq: invariance}
a^{i}\cdot u\in \mathcal{N}\subseteq \bigcap_{a\in A}L\left((\partial\mathrsfs{A})^{-},a\right)
\end{equation}
for all $i\in [1,m(a)]$. In particular, $u=a\cdot (a^{m(a)-1}\cdot u)$ with $a^{m(a)-1}\cdot u\in L\left((\partial\mathrsfs{A})^{-},a\right)$, and so by (\ref{eq: charact language}) we get $u\in L\left([(\partial\mathrsfs{A})^{-}]^{-1},a^{-1}\right)$. Whence, by repeating the previous argument for all the elements of $A^{-1}$, we deduce that (\ref{eq: inclusion to prove}) holds. The invariance of $u$ for the action $A^{-1}\overset{\cdot}{\curvearrowright} \wt{Q}^{*}$ follows from (\ref{eq: invariance}) since for all $a\in A$ we have $a^{-1}\cdot u=a^{m(a)-1}\cdot u\in\mathcal{N}$.
\end{proof}
\noindent Therefore, while the algebraic structure of $\mathcal{G}(\mathrsfs{A})$ is unchanged when we enrich it, the dynamics of $\mathcal{G}(\mathrsfs{A})$ on $\wt{Q}^{\omega}$ is ``richer'' than on $Q^{\omega}$, and as we will see, part of this dynamics is influencing the algebraic structure of the group generated by its dual.

Theorem \ref{theo: characterization free} shows that the property of not being a free group can be characterized in term of existence of finite orbital graphs centered at periodic elements in the boundary of $\wt{Q}^{\omega}$ under the action of the semigroup $S=\mathcal{S}(\partial \mathrsfs{A}^{-})$. Since $\mathrsfs{A}$ is bireversible, the orbital graph $\Gamma(S,A,\wt{Q}^{\omega},\ul{v})$ ($\Gamma(S,A,\wt{Q}^{*},u)$) is isomorphic (considering only the positive edges) to the Schreier graph of the stabilizer of $\ul{v}$ ($u$) for the group $\mathcal{G}(\partial \mathrsfs{A} ^{-})$; henceforth we denote it by $\Sch(\ul{v})$ ($\Sch(u)$). We now expand Theorem \ref{theo: characterization free} to the bireversible case by linking the property of not being a free group for $\mathcal{G}(\mathrsfs{A})$ to the existence of finite Schreier graphs in the boundary $\wt{Q}^{\omega}$ under the action of the group $\mathcal{G}(\partial \mathrsfs{A}^{-})$. It is clear that just the existence of a finite Schreier graph is not enough, indeed there are always finite Schreier graphs (no matter $\mathcal{G}(\mathrsfs{A})$ is free or not), namely all $\Sch(y^{\omega})$, with $\oo{y}=1$, are finite. Therefore, we need to consider a smaller class of points $\wt{Q}^{\omega}$ in the boundary: the class of \emph{essentially non-trivial elements}. Roughly speaking these points represent elements in the boundary of $F_{Q}$. More precisely, for two given elements $u,v \in F_Q$, written in reduced form, we denote by $u\wedge v$ the longest common prefix of $u$ and $v$. Thus, $(F_Q,d)$ where $d(u,v) = 2^{-r(u,v)}$ with
$$
r(u,v) = \left\{
\begin{array}{ll}
|u\wedge v|&\mbox{ if }u \neq v\\
+\infty&\mbox{ otherwise}
\end{array}
\right.
$$
is a metric space which is not complete, but its completion admits a simple description: we add to $F_Q$ all the (right) infinite reduced words $q_1q_2q_3\ldots$ on $\wt{Q}^{\omega}$. These new elements are called the {\em boundary} of $F_Q$ and the completion (which is indeed compact) is denoted by $\widehat{F_Q}$. For every element $\ul{u}\in \wt{Q}^{\omega}$ we can associate the following sequence
$$
v_{n}= \oo{\ul{u}[n]}
$$
of elements in $\widehat{F_Q}$. Since this space is metric and compact, every sequence has a converging subsequence. Thus, we have the following definition.
\begin{definition}[essentially non-trivial]
An element $\ul{u}\in\wt{Q}^{\omega}$ is called essentially non-trivial if there exists a convergent subsequence $\{v_{n_{j}}\}_{j>0}$ of $\{ \oo{\ul{u}[n]}\}_{n>0}$ which converges to an element of $\widehat{F_Q}\setminus F_{Q}$.
\end{definition}
Words that are not essentially non-trivial are called \emph{essentially trivial}. The following proposition gives another description of these points.
\begin{proposition}\label{prop: charact ess trivial}
Let $\ul{u}\in\wt{Q}^{\omega}$, the following are equivalent:
\begin{itemize}
\item[i)] $\ul{u}$ is essentially trivial;
\item[ii)] There is an integer $m$ such that $\{ \oo{\ul{u}[n]}\}_{n>0}\subseteq \wt{Q}^{\le m}$;
\item[iii)] There is an integer $m$ such that for any finite factor $w$ of $\ul{u}\in\wt{Q}^{\omega}$, $\oo{w}\in \wt{Q}^{\le m}$;
\item[iv)] $\ul{u}$ is the label of a (right) infinite path on a finite rooted $Q$-inverse tree $(\mathcal{T},r)$ starting from the root $r$.
\end{itemize}
\end{proposition}
\begin{proof}
$i)\Leftrightarrow ii)$. Trivial.\\
$ii)\Leftrightarrow iii)$. Trivial.\\
$ii)\Rightarrow iv)$. Since $\{ \oo{\ul{u}[n]}\}_{n>0}\subseteq \wt{Q}^{\le m}$ the set $R$ of group-reduced left factors of $\ul{u}$ is finite. This set gives rise to a finite rooted inverse tree $(\mathcal{M}(R), r)$ on $Q$ with root $r$, such that $u[n]$ is the label of a path in $\mathcal{M}(R)$ starting from $r$. Hence, $\ul{u}$ is the label of a (right) infinite path in $\mathcal{M}(R)$ starting from $r$. \\
$iv)\Rightarrow ii)$. Note that for a finite rooted inverse tree $(\mathcal{T},r)$, a vertex $v$ of $(\mathcal{T},r)$ is one-to-one correspondence with the shortest reduced word $u$ such that $r\vlongmapright{u}v$, and if there is a path $r\vlongmapright{w}v$, then $u=\oo{w}$. Hence, $\ul{u}[n]$ is contained in the set of vertices of $(\mathcal{T},r)$, i.e. $\{ \oo{\ul{u}[n]}\}_{n>0}$ is finite.
\end{proof}
\noindent We now show that essentially trivial points of $\wt{Q}^{\omega}$ are points in which the Schreier graphs are always finite. Therefore, the dynamics of $\mathcal{G}(\partial \mathrsfs{A} ^{-})$ on them is trivial, and as we will see they do not influence the algebraic structure of $\mathcal{G}(\mathrsfs{A})$. This is proven in the next theorem, but first we need some lemmata.
\begin{lemma}\label{lem: image of complete}
Let $ \mathrsfs{B}=(T,B,\cdot,\circ)$ be a bireversible transducer, if $\mathcal{D}=(\Gamma,v)$ is a complete inverse automaton on $B$, then for any $q\in T$ the language $\mathrsfs{B}_{q}(L[\mathcal{D}])$ is also defined by a complete inverse automaton $\mathcal{C}=(\Lambda,r)$ on $B$.
\end{lemma}
\begin{proof}
To the complete inverse automaton $\mathcal{D}$ we can associate an inverse transducer $\mathrsfs{D}$ by adding the identity output:
$$
\forall b\in\wt{B}, q\mapright{b|b}q' \mbox{ is an edge in } \mathrsfs{D}\mbox{ if and only if }  q\mapright{b}q'\mbox{ is an edge in } \mathcal{D}
$$
Therefore, by a simple computation with products of transducers one gets
$$
\mathrsfs{B}_{q}(L[\mathcal{D}])=\left((\mathrsfs{D}\mathrsfs{B})_{\mathcal{O}}, (v,q)\right)
$$
Since both $\mathrsfs{D}$ and $\mathrsfs{B}$ are bireversible inverse transducers, then by \cite[Proposition 2]{DaRo14} $\mathrsfs{D}\mathrsfs{B}$ is also a bireversible inverse transducer. Hence, $(\Lambda,r)=\left((\mathrsfs{D}\mathrsfs{B})_{\mathcal{O}}, (v,q)\right)$ is a complete inverse automaton recognizing $\mathrsfs{B}_{q}(L[\mathcal{D}])$, and this concludes the proof.
\end{proof}
\noindent Recall that $R_{T}$ is the set of reduced words on $\wt{T}^{*}$, we have the following lemma.
\begin{lemma}\label{lem: finite index preimage}
Let $ \mathrsfs{B}=(T,B,\cdot,\circ)$ be a bireversible automaton, let $H\le F_{T}$ be a finite index subgroup, and let $u\in B^{*}$. Then the set $\{w\in R_{T}: w\cdot u\in H\}$ is also a finite index subgroup of $F_{T}$ .
\end{lemma}
\begin{proof}
Let $ \partial\mathrsfs{B}=(B,T,\circ, \cdot)$, by Proposition \ref{prop: dual prop} it is bireversible. Hence, $\partial\mathrsfs{B}^{-}$ is also bireversible, and so its inverse $\mathrsfs{C}=(\partial\mathrsfs{B}^{-})\inv$ is a bireversible inverse transducer. It is not difficult to check that, by the bireversibility, $\mathrsfs{C}$ sends reduced words into reduced words, whence for any language $L\subseteq \wt{T}^{*}$, $v\in (B^{-1})^{*}$ we have
$$
\mathrsfs{C}_{v}(\oo{L})=\oo{\mathrsfs{C}_{v}(L)}
$$
Let $u=t_{m}\ldots t_{1}$, if we put $v=t_{m}^{-1}\ldots t_{1}^{-1}$, then by a simple computation we get
$$
\mathrsfs{C}_{v}(H^{R})^{R}=\{w\in R_{T}: w\cdot u\in H\}
$$
where $H^{R}=\{h^{R}:h\in H\}$. Let us prove that $\mathrsfs{C}_{v}(H^{R})^{R}$ is a finite index subgroup of $F_{T}$ by induction on the length of the word $v$. Note that for any finitely generated subgroup $D\le F_{T}$, since the Stallings automaton $\mathcal{S}_{t}(D)$ is inverse, then $\mathcal{S}_{t}(D^{R})$ is obtained by $\mathcal{S}_{t}(D)$ by simply reversing the direction of the edges, whence $D$ has finite index if and only if $D^{R}$ has finite index. Therefore, it is enough to show that $\mathrsfs{C}_{v}(H)$ is a finite index subgroup of $F_{T}$. Since $H$ has finite index, its Stallings automaton $\mathcal{S}_{t}(H)$ is a complete inverse $T$-automaton.
Furthermore, by Lemma \ref{lem: image of complete} $\mathrsfs{C}_{t_{1}^{-1}}(L[\mathcal{S}_{t}(H)])$ is defined by a complete inverse $T$-automaton $(\Lambda, r)$. Thus,
$$
\mathrsfs{C}_{t_{1}^{-1}}(H)=\mathrsfs{C}_{t_{1}^{-1}}(\oo{L[\mathcal{S}_{t}(H)]})=\oo{\mathrsfs{C}_{t_{1}^{-1}}(L[\mathcal{S}_{t}(H)])}=\oo{L[(\Lambda, r)]}=H_{0}
$$
is a finite index subgroup of $F_{T}$, which proves the base case of our claim. Furthermore, using the induction hypothesis, we get that
$$
\mathrsfs{C}_{v}(H)=\mathrsfs{C}_{t_{m}^{-1}\ldots t_{2}^{-1}}\left (\mathrsfs{C}_{t_{1}^{-1}}(H)\right )=\mathrsfs{C}_{t_{m}^{-1}\ldots t_{2}^{-1}}\left (H_{0}\right )
$$
is also a finite index subgroup of $F_{T}$, and this concludes the proof of the lemma.
\end{proof}

\begin{theorem}\label{theo: finiteness essentially trivial}
For any essentially trivial point $\ul{v}\in\wt{Q}^{*}$, $\Sch(\ul{v})$ is finite.
\end{theorem}
\begin{proof}
Let $\partial \mathrsfs{A} ^{-}=(A,\wt{Q},\circ, \cdot)$, and put $D=\mathcal{G}(\partial \mathrsfs{A} ^{-})$, and let $\pi:F_{Q}\rightarrow D$ be the canonical map. By Proposition \ref{prop: charact ess trivial} it is enough to prove that, given a rooted tree $(\mathcal{T},r)$ there is a finite index subgroup $H\le F_{Q}$ such that, for any (right) infinite word $\ul{u}\in  \wt{Q}^{\omega}$ labeling a (right) infinite path staring from $r$, we have $H\subseteq \pi^{-1}( \St_{D}(\ul{u}))$. We first prove the following claim:
\begin{itemize}
\item[\textbf{C}:] for any vertex $v$ of $\mathcal{T}$, there is a finite index subgroup $H_{v}\le F_{Q}$ such that $H_{v}\subseteq \pi^{-1}( \St_{D}(u))$ for any $u\in L[(\mathcal{T},v)]$.
\end{itemize}
We prove this claim using an induction on the number of vertices of $\mathcal{T}$. Indeed, if $\mathcal{T}$ consists of just two vertices $v,v'$ connected by an edge $v\mapright{q}v'$, $q\in\wt{Q}$, then for any $u\in L[(\mathcal{T},v)]$, $u=(qq^{-1})^{m}$ for some $m\ge 0$, whence in this case it is clear that $\pi^{-1}(\St_{D}(q))$ is a finite index subgroup stabilizing $(qq^{-1})^{m}$ for any $m\ge 0$. Hence, in this case our statement holds. Therefore, we can assume that $\mathcal{T}$ has more than two vertices. We consider the following two cases:
\begin{itemize}
\item We have two distinct edges $v\mapright{q}v'$, $v\mapright{p}v''$, $q,p\in \wt{Q}$. In this case, we may consider any two maximal sub-trees $\mathcal{T}_{1},\mathcal{T}_{2}$ of $\mathcal{T}$ having just the vertex $v$ in common. Thus, any $u\in L[(\mathcal{T},v)]$ has a unique decomposition of alternating elements of $L[(\mathcal{T}_{1},v)]$ and $ L[(\mathcal{T}_{2},v)]$, i.e.
\begin{equation}\label{eq: alternating fact}
u=u_{1}\ldots u_{m}
\end{equation}
such that if $u_{i}\in L[(\mathcal{T}_{j},v)]$, then $u_{i+1}\in L[(\mathcal{T}_{3-j},v)]$ for all $i=1,\ldots, m-1$ and some $j\in\{1,2\}$. Since the number of vertices of  $\mathcal{T}_{1}$ and  $\mathcal{T}_{2}$ is strictly less then the one of $\mathcal{T}$, by the induction hypothesis we get that there are finite index subgroups $K_{1},K_{2}\le F_{Q}$ stabilizing all the elements in $L[(\mathcal{T}_{1},v)]$, $ L[(\mathcal{T}_{2},v)]$, respectively. Since each element $u_{i}$ appearing in the factorization (\ref{eq: alternating fact}) satisfies $\oo{u_{i}}=1$, then it is straightforward to check that $K_{1}\cap K_{2}$ is a finite index subgroup of $F_{Q}$ stabilizing $u$, and therefore it stabilizes all the elements in $ L[(\mathcal{T},v)]$.
\item Assume that there is just one edge $v\mapright{q}v'$ in $\mathcal{T}$. Since there are more then two vertices, there is an edge $v'\mapright{p}v''$, for some $p\in\wt{Q}\setminus\{q^{-1}\}$. Consider the sub-tree $\mathcal{T}'$ of $\mathcal{T}$, obtained from $\mathcal{T}$ by erasing the two edges $v\mapright{q}v'$, $v'\mapright{q\inv}v$. It is not difficult to check that for any $u\in L[(\mathcal{T},v)]$ we have the following unique factorization:
\begin{equation}\label{eq: alternating fact 2}
u=(qu_{1}q^{-1})\ldots (qu_{m}q^{-1})
\end{equation}
for some $u_{i}\in L[(\mathcal{T}',v')]$, $i=1,\ldots, m$. Since the number of vertices of $\mathcal{T}'$ is strictly less then the one of $\mathcal{T}$, by the induction hypothesis we get that there is a finite index subgroup $K\le F_{Q}$ stabilizing the elements of $ L[(\mathcal{T}',v')]$. Therefore, since each $u_{i}$ appearing in (\ref{eq: alternating fact 2}) satisfies $\oo{u_{i}}=1$, it is straightforward to check that the set:
$$
M=\{w\in R_{Q}: w\circ q\in K\}\cap \pi^{-1}(\St_{D}(q))
$$
stabilizes all the elements in $ L[(\mathcal{T},v)]$, and by Lemma \ref{lem: finite index preimage} $M\le F_{Q}$ has finite index, and this concludes the proof of the claim \textbf{C}.
\end{itemize}
Fixed a root $r$, any vertex $v$ of $\mathcal{T}$ is identified with the unique reduced word $v$ connecting $r$ to $v$, let $V(\mathcal{T}, r)$ denote this set, and for any $v\in V(\mathcal{T})$, $H_{v}$ denotes the finite index subgroup of claim \textbf{C}. By Lemma \ref{lem: finite index preimage}
$$
H=\bigcap_{u\in V(\mathcal{T})}\{w\in R_{Q}: w\circ u\in H_{v} \}
$$
is also a finite index subgroup of $F_{Q}$. We show that for any path $r\mapright{u}v$, $u\in \wt{Q}^{*}$, in $\mathcal{T}$, $H$ stabilizes $u$. This fact clearly implies that $H\subseteq \pi^{-1}( \St_{D}(\ul{u}))$ for any infinite word $\ul{u}\in  \wt{Q}^{\omega}$ labeling a (right) infinite path staring from $r$. Consider the unique reduce path in $\mathcal{T}$:
$$
r\mapright{u}v=r=v_{0}\mapright{x_{1}}v_{1}\mapright{x_{2}}\ldots v_{i}\mapright{x_{i}}v_{i+1}\mapright{x_{i+1}}\ldots v_{n-1}\mapright{x_{n}}v_{n}
$$
It is not difficult to check that we have the following factorization
$$
u=u_{0}x_{1}u_{1}x_{2}\ldots x_{i}u_{i}\ldots x_{n}u_{n}
$$
where $u_{i}\in L[(\mathcal{T},v_{i})]$. Since $\oo{u_{i}}=1$ it is straightforward to check that $u$ is stabilized by
$$
H_{v_{0}}\cap \left(\bigcap_{j=1}^{n}\left \{w\in R_{Q}: w\circ (x_{1}\ldots x_{j})\in H_{x_{0}\ldots x_{j}}\right \}\right)
$$
and this set contains $H$, whence $H$ stabilizes $u$ and this concludes the proof of the proposition.
\end{proof}
Since the action of $\mathcal{G}(\partial \mathrsfs{A}^{-})$ on $\wt{Q}^{\omega}$ is essentially-free (being $\partial \mathrsfs{A}^{-}$ bireversible) \cite{StVoVo2011}, i.e.
$$
m(\{x\in \wt{Q}^{\omega}: \ \St_{\mathcal{G}(\partial \mathrsfs{A}^{-})}(x)\neq \id \})=0,
$$
where $m$ is the uniform measure on $\wt{Q}^{\omega}$ (see \cite{DynSubgroup}), by Theorem \ref{theo: finiteness essentially trivial}, and the fact that there are bireversible transducers defining infinite groups generated by a set $Q$ with $|Q|\ge 2$, we immediately derive the following corollary.
\begin{corollary}
The set of essentially trivial points on $\wt{Q}^{\omega}$, $|Q|\ge 2$, has zero measure with respect to the uniform measure on $\wt{Q}^{\omega}$.
\end{corollary}
Excluding essentially free points, whose dynamics is trivial, we now show that the dynamics on the complement of this set (the non-trivial points) is enough to characterize the freeness of the group $\mathcal{G}(\mathrsfs{A})$. We devote the rest of this section to prove this result.
\\
The following proposition specializes Lemma \ref{lem: periodic in finite schr} to the bireversibile case.
\begin{proposition}\label{prop: almost implies periodic}
Let $\mathrsfs{A}=(Q,A,\cdot,\circ)$ be a bireversible automaton, and let $S=\mathcal{S}(\partial \mathrsfs{A}^{-})$. Suppose that $\|\Gamma(S,A,\wt{Q}^{\omega},xy^{\omega})\|<\infty$ for some $x,y\in \wt{Q}^{*}$, hence there is an integer $n$ such that
$$
\|\Sch(y^{\omega})\|=\|\Gamma(S,A,\wt{Q}^{\omega},y^{\omega})\|<\infty
$$
\end{proposition}
\begin{proof}
By Lemma \ref{lem: periodic in finite schr}, if $n$ is the least common multiple of the integers $k(z)$ for each vertex $hz^{\omega}$ of $\Gamma(S,A,\wt{Q}^{\omega},xy^{\omega})$, then for any $u\in A^{*}$ we have that $u\cdot x(y^{n})^{\omega}=x'z^{\omega}$ with $|x|=|x'|$ and $|z|=n|y|$. Since $\mathrsfs{A}$ is bireversible it is straightforward to check that for any $w\in A^{*}$ there is a $u\in A^{*}$ such that $u\cdot x(y^{n})^{\omega}=x'w\cdot (y^{n})^{\omega}$, whence
$$
x'z^{\omega}=u\cdot x(y^{n})^{\omega}=x'w\cdot (y^{n})^{\omega}
$$
Thus, for any $w\in A^{*}$, $w\cdot (y^{n})^{\omega}=z^{\omega}$, for some $z\in \wt{Q}^{*}$ with $|z|=n|y|$, whence the claim $\|\Gamma(S,A,\wt{Q}^{\omega},y^{\omega})\|<\infty$.
\end{proof}
Note that the Schreier graphs of $\partial \mathrsfs{A} ^{-}$ in the boundary can be seen as limit of the components in the powers of the dual of $\partial \mathrsfs{A} ^{-}$. Since $\partial(\partial \mathrsfs{A}^{-})= \mathrsfs{A}\sqcup \mathrsfs{A}^{-1}$, if we put $\wt{\mathrsfs{A}}=\partial(\partial \mathrsfs{A}^{-})= \mathrsfs{A}\sqcup \mathrsfs{A}^{-1}$, then we may consider the following growth function:
$$
\chi_{\wt{\mathrsfs{A}}}(n)=\min\left\{\left\|\left(\wt{\mathrsfs{A}}^{n}, q\right)\right \|: q\in \wt{Q}^{n}\mbox{ with }\oo{q}=q\right\}
$$
This function is monotonically increasing. The following theorem links the non-freeness of $\mathcal{G}(\mathrsfs{A})$ with both the growth of $\chi_{\wt{\mathrsfs{A}}}(n)$ and the dynamics of $\partial \mathrsfs{A} ^{-}$ on the essentially non-trivial points of the boundary $\wt{Q}^{\omega}$.
\begin{theorem}\label{theo: finiteness schreier}
Let $\mathrsfs{A}=(Q,A,\cdot,\circ)$ be a bireversible transducer. Let $m\ge 1$ be a positive integer. Then, the following are equivalent
\begin{enumerate}
  \item [i)]There is a finite Schreier graph $ \Sch(\ul{v})$, for some essentially non-trivial element $\ul{v}\in \wt{Q}^{\omega}$;
  \item[ii)] There exists an almost periodic element $\ul{v}=xy^{\omega}$ with $\oo{xy}=xy$ such that $|x|+|y|\le m+(2^{m}|A|\|\mathrsfs{A}\|^{m})^{|A|}$ and $\|\Sch(\ul{v})  \|= \chi_{\wt{\mathrsfs{A}}}(m)$;
  \item[iii)] $\chi_{\wt{\mathrsfs{A}}}(m)$=$\chi_{\wt{\mathrsfs{A}}}(m+i)$ for all $ i\le  (2^{m}|A|\|\mathrsfs{A}\|^{m})^{|A|}$;
  \item[iv)] $\chi_{\wt{\mathrsfs{A}}}(m)$=$\chi_{\wt{\mathrsfs{A}}}(m+i)$ for all $i\ge 0$;
  \item[v)] There exists a periodic element $y^{\omega}$ with $\oo{y}=y$ such that $\|\Sch(y^{\omega})  \|<\infty$;
  \item[vi)] $\mathcal{G}(\mathrsfs{A})=F_{Q}/N$ is not free with a non-trivial relation $y^{r}\in\wt{Q}^{*}$ for some positive integer $r$.
\end{enumerate}
\end{theorem}
\begin{proof}
$i)\Rightarrow iv)$.
Using the definitions it is straightforward to check that the following relation between Schreiers graphs:
\begin{equation}\label{eq: upper reduced}
\|\Sch(\oo{u})\|\le \|\Sch(u)\|
\end{equation}
holds for all $u\in\wt{Q}^{*}$. Since $\chi_{\wt{\mathrsfs{A}}}$ is monotonically increasing, if $\lim_{i\to\infty}\chi_{\wt{\mathrsfs{A}}}(m+i)=\infty$, then by Theorem \ref{thm:schreier} for each infinite reduced word $\ul{u}\in \wt{Q}^{\omega}$ we have $\|\Sch(\ul{u})\|=\infty$. We now prove that for any essentially non-trivial element $\ul{v}$, $\| \Sch(\ul{v})\|=\infty$. Let $\{u_{n_{i}}\}_{i>0}$ be the subsequence of $\{\oo{\ul{v}[n]}\}_{n>0}$ converging to a reduced element $\ul{u}$ of $\widehat{F_Q}\setminus F_{Q}$. Since $\oo{\ul{v}[n_{i}]}=u_{n_{i}}$, then by (\ref{eq: upper reduced})
$$
\| \Sch(u_{n_{i}})\|\le \| \Sch(\ul{v}[n_{i}])\|
$$
for all $i>0$. Since $\lim_{i\to \infty} u_{n_{i}}=\ul{u}\in \widehat{F_Q}\setminus F_{Q}$ and $u$ is reduce, then $\|\Sch(\ul{u})\|=\infty$, whence the claim
$$
\| \Sch(\ul{v})\|\ge \lim_{i\to \infty } \| \Sch(\ul{v}[n_{i}])\|\ge \lim_{i\to \infty }\| \Sch(u_{n_{i}})\|\ge\infty
$$
$iv)\Rightarrow iii)$. Trivial.\\
$iii)\Rightarrow ii)$. Condition $iii)$ implies that there is a connected component
$$
\left(\wt{\mathrsfs{A}}^{m}, (q_{1},\ldots,q_{m})\right)
$$
of cardinality $\chi_{\wt{\mathrsfs{A}}}(m)$ and a finite sequence $q_{m+i}$ of vertices of $\wt{\mathrsfs{A}}$, with $0\le i\le  (2^{m}|A|\|\mathrsfs{A}\|^{m})^{|A|}$, such that for all $0\le i\le  (2^{m}|A|\|\mathrsfs{A}\|^{m})^{|A|}$
$$
\|(\wt{\mathrsfs{A}}^{m+i}, q_{1}\ldots q_{m+i}))\|=\chi_{\wt{\mathrsfs{A}}}(m)
$$
and $\oo{q_{1}\ldots q_{m+i}}=q_{1}\ldots q_{m+i}$. Hence, by Propositions \ref{prop: immersion lang hom}, \ref{prop: inverse of products}, and \cite[Proposition 7 ]{DaRo14} this sequence of transducers have all isomorphic input automata:
\begin{equation}\label{eq: all isom input auto}
\left((\wt{\mathrsfs{A}}^{m+i})_{\mathcal{I}}, q_{1}\ldots q_{m+i}\right)\simeq \left((\wt{\mathrsfs{A}}^{m+i+1})_{\mathcal{I}}, q_{1}\ldots q_{m+i+1}\right)
\end{equation}
Since there are at most $(2^{m}|A|\|\mathrsfs{A}\|^{m})^{|A|}$ possible transducers with the same input automaton $\left((\wt{\mathrsfs{A}}^{m})_{\mathcal{I}}, q_{1}\ldots q_{m}\right)$, there are two integers $k,p$ with $k+p\le (2^{m}|A|\|\mathrsfs{A}\|^{m})^{|A|}$ such that the following isomorphism (as $\wt{A}\times\wt{A}$-automata)
\begin{equation}\label{eq: repetition as transducers}
\left(\wt{\mathrsfs{A}}^{m+k}, q_{1} \ldots q_{m+k}\right)\simeq \left(\wt{\mathrsfs{A}}^{m+k+p}, q_{1} \ldots q_{m+k+p}\right)
\end{equation}
holds. Consider the words $x=q_{1}\ldots q_{m+k}$, $y=q_{m+k+1}\ldots q_{m+k+p}$, and the almost periodic point $\ul{v}=xy^{\omega}$. Hence, $\oo{xy}=xy$, $|x|+|y|\le m+(2^{m}|A|\|\mathrsfs{A}\|^{m})^{|A|}$, and by (\ref{eq: repetition as transducers}) it is not difficult to check that
$$
\left(\wt{\mathrsfs{A}}^{m+k+p}, xy\right)\simeq \left(\wt{\mathrsfs{A}}^{m+k+tp}, xy^{t}\right)
$$
holds for any $t\ge 1$. Whence by Theorem \ref{thm:schreier} we obtain
$$
 \left((\wt{\mathrsfs{A}}^{m+k})_{\mathcal{I}}, x\right)^{-}\simeq \left(\wt{\mathrsfs{A}}^{\infty}_{\mathcal{I}}, xy^{\omega}\right)^{-}\simeq \Sch(xy^{\omega})
$$
is a finite Schreier graph with $\chi_{\wt{\mathrsfs{A}}}(m)$ vertices.\\
$ii)\Rightarrow i)$. It is enough to show that $xy^{\omega}$ is essentially non-free. Indeed, let $u\in\wt{Q}^{*}$ be the maximal prefix of $y$ such that $y=uzu^{-1}$. Since $\oo{xy}=xy$ it is straightforward to check that in $\{\oo{xy^{\omega}[n]}\}_{n>0}$ there is the subsequence $\{xuz^{m}\}_{m>0}$ of reduced words. Furthermore, it is obvious that $\lim_{m\to\infty}xuz^{m}=xuz^{\omega}\in \widehat{F_Q}\setminus F_{Q}$, i.e., $xy^{\omega}$ is essentially non-free.\\
$v)\Rightarrow i)$. Analogously as above, $y^{\omega}$ is essentially non-trivial.\\
$ii)\Rightarrow v)$. Since $\oo{xy}=xy$ we get $\oo{y}=y$. Furthermore, by Proposition \ref{prop: almost implies periodic} $\|\Sch(xy^{\omega}) \|<\infty$ implies $\|\Sch(y^{\omega}) \|<\infty$. \\
$v)\Leftrightarrow iv)$. Theorem \ref{theo: characterization free}.
\end{proof}

\subsection{Some general consequences}
In this section we derive some general consequences of Theorem \ref{theo: characterization free} and Theorem \ref{theo: finiteness schreier}. Note that Theorem \ref{theo: characterization free} shows that a relation $y^r=\id$ in $\mathcal{G}(\mathrsfs{A})$ corresponds to a finite orbital graph $\Gamma(S,A,\wt{Q}^{\omega},y^{\omega})$ for the action of the dual transducer. This enables us to give an elementary proof of the following known fact that can be also derived using Zelmanov's celebrated theorem \cite{Zelma90,Zelma91}.

\begin{corollary}
Let $\mathrsfs{A}=(Q,A,\cdot,\circ)$ be an invertible automaton, and let $G:=\mathcal{G}(\mathrsfs{A})$. If $G$ is a Burnside group with bounded exponent, then $G$ is finite.
\end{corollary}
\begin{proof}
By assumption there exists $N\in \mathbb{N}$ such that for every $g\in G$, $g^N=\id$. Thus it is easy to see that the cardinality of the sets $Y_{z}=\{a\cdot z^{N}: a\in A^{\ast}\}$, $z\in \wt{Q}^{*}$, are finite and uniformly bounded by some constant $M>0$. Hence Theorem \ref{theo: characterization free} yields that the cardinality of the orbital graph of $z^{\omega}$ is bounded by $M$. The claim follows if we prove that all orbital graphs of elements $\ul{u}\in \widetilde{Q}^{\omega}$ have a number of vertices less than $M$ \cite[Corollary 1]{DaRo14}. Suppose that this is not the case, and there exists a non-periodic element $\ul{u}=u_1u_2\cdots\in \widetilde{Q}^{\omega}$ such that $\|\Gamma(S,A,\wt{Q}^{\omega},\ul{u}) \|>M$. This implies that there is a large enough $n=n(M)$ such that the orbit of the prefix $\ul{u}[n]=u_1\cdots u_n$ under the action of the dual transducer contains more that $M$ elements. This is absurd, because
$$
M\geq \|\Gamma(S,A,\wt{Q}^{\omega},\ul{u}[n]^{\omega}) \|>\|\Gamma(S,A,\wt{Q}^{\omega},\ul{u}[n]) \|>M
$$
\end{proof}
Let $\mathrsfs{A}=(Q,A,\cdot,\circ)$ be an invertible automaton, and let $\mathcal{G}(\mathrsfs{A})=F_{Q}/N$. The set of ``positive relations'' of $\mathcal{G}(\mathrsfs{A})$ is given by the set
$$
\mathcal{P}(\mathrsfs{A})=Q^{+}\cap \sigma^{-1}(N)
$$
where we recall that $\sigma:\wt{Q}^{*}\rightarrow F_{Q}$ is the canonical homomorphism. Note that $\mathcal{P}(\mathrsfs{A})=\emptyset$ implies that $\mathcal{S}(\mathrsfs{A})$ is torsion-free and therefore infinite. The following two corollaries of Theorem \ref{theo: characterization free} give some sufficient conditions not to have any positive relations.
\begin{corollary}\label{cor: trans implies positive free}
If $\mathcal{S}(\partial\mathrsfs{A})$ acts spherically transitively on $Q^{*}$, then $\mathcal{P}(\mathrsfs{A})=\emptyset$.
\end{corollary}
\begin{proof}
If $\mathcal{S}(\partial\mathrsfs{A})$ acts spherically transitive on $Q^{*}$, then $\|\Gamma(\mathcal{S}(\partial\mathrsfs{A}),A,Q^{\omega},y^{\omega})\|=\infty$, for all $y\in Q^{+}$. Since the action of the semigroup $S=\mathcal{S}(\partial \mathrsfs{A}^{-})$ on $\wt{Q}^{\omega}$ restricted on $Q^{\omega}$ is the same as that of $\mathcal{S}(\partial\mathrsfs{A})$ on $Q^{\omega}$, we get $ \|\Gamma(S,A,\wt{Q}^{\omega},y^{\omega}) \|=\infty$, for all $y\in Q^{+}$. Hence by Theorem \ref{theo: characterization free} we get $\mathcal{P}(\mathrsfs{A})=\emptyset$.
\end{proof}
\noindent If $\mathrsfs{A}$ is a reversible invertible transducer (for short $RI$-transducer), then $\partial\mathrsfs{A}$ is an $RI$-transducer by Proposition \ref{prop: dual prop}, therefore it defines a group, hence in this case $\mathcal{S}(\partial\mathrsfs{A})$ acts spherically transitive on $Q^{*}$ if and only if $\mathcal{G}(\partial\mathrsfs{A})$ acts spherically transitive on $Q^{*}$. In this class we obtain this following sufficient condition not to have any positive relations.
\begin{corollary}\label{cor: not containing bireversible}
Let $\mathrsfs{A}$ be an $RI$-transducer. If $\mathrsfs{A}$ does not contain a bireversible connected component, then $\mathcal{P}(\mathrsfs{A})=\emptyset$.
\end{corollary}
\begin{proof}
Note that by \cite[Proposition 2]{DaRo14} if $\mathrsfs{B}$ and $\mathrsfs{C}$ are invertible, then $\mathrsfs{B}\mathrsfs{C}$ is bireversible if and only if $\mathrsfs{B}$ and $\mathrsfs{C}$ are bireversible. Thus, the condition in the statement ensures that for any $k\ge 1$, $\mathrsfs{A}^{k}$ does not contain any bireversible connected component. Assume, contrary to the claim, that there is a $y\in \mathcal{P}(\mathrsfs{A})$. Let $S=\mathcal{S}(\partial\mathrsfs{A})$, $D=\mathcal{G}(\partial\mathrsfs{A})$. By Theorem \ref{theo: characterization free} $\|\Gamma(S,A,Q^{\omega},y^{\omega})\|<\infty$. Hence, by Theorem \ref{thm:schreier} we get
\begin{align}
\nonumber \left\|\left((\mathrsfs{A}^{-})^{\infty}_{\mathcal{I}},y^{\omega}\right)\right\|&=\left\|\left(\Sch_{D}(\St_{D}(y^{\omega}),\wt{A}),\St_{D}(y^{\omega})\right)\right\|=\|\Gamma(D,\wt{A},Q^{\omega},y^{\omega})\|=\\
\nonumber&= \|\Gamma(S,A,Q^{\omega},y^{\omega})\|<\infty
\end{align}
In particular, the monotonically increasing sequence $\{\|((\mathrsfs{A}^{-})^{2^{k}|y|}_{\mathcal{I}},y^{2^{k}})\|\}_{k\ge 1}$ stabilizes. Hence, by \cite[Proposition 8]{DaRo14} we have that there is an $k\ge 1$ such that $L((\mathrsfs{A}^{-})^{2^{k}|y|}_{\mathcal{O}},y^{2^{k}})= L((\mathrsfs{A}^{-})^{2^{k}|y|}_{\mathcal{I}},y^{2^{k}})$ (with the notation introduced in \cite[Section 5]{DaRo14},  $((\mathrsfs{A}^{-})^{2^{k}|y|}_{\mathcal{I}},y^{m2^{k}})$ has the swapping invariant property with respect to the pair $(y^{2^{k}},y^{2^{k}})$). However, by \cite[Proposition 9]{DaRo14} we get $((\mathrsfs{A}^{-})^{2^{k}|y|}_{\mathcal{O}},y^{2^{k}})\simeq  ((\mathrsfs{A}^{-})^{2^{k}|y|}_{\mathcal{I}},y^{2^{k}})$, whence both $((\mathrsfs{A})^{2^{k}|y|}_{\mathcal{O}},y^{2^{k}})$ and $((\mathrsfs{A})^{2^{k}|y|}_{\mathcal{I}},y^{2^{k}})$ is reversible, i.e $((\mathrsfs{A})^{2^{k}|y|},y^{2^{k}})$ is bireversible, a contradiction.
\end{proof}
The following result extends \cite[Corollary 5]{DaRo14} in the class of $RI$-transducers and it restricts the search for transducers with trivial stabilizers to the class of $RI$-transducers whose duals have no positive relations.
\begin{corollary}\label{cor: no positive dual}
Let $\mathrsfs{A}=(Q,A,\cdot, \circ)$ be an $RI$-transducer such that $G=\mathcal{G}(\mathrsfs{A})$ is infinite. Then the index $\left[G:\St_{G}(y^{\omega})\right]$ is infinite for all $y\in A^{*}$, if and only if $\mathcal{P}(\partial\mathrsfs{A})=\emptyset$. In particular, if $\St_{G}(\ul{u})=\id$ for all $\ul{u}\in A^{\omega}$, then $\mathcal{P}(\partial\mathrsfs{A})=\emptyset$.
\end{corollary}
\begin{proof}
Similarly to the previous proof, let $S=\mathcal{S}(\mathrsfs{A})$. By Theorem \ref{theo: characterization free}, Theorem \ref{thm:schreier} and the fact that $G$ is infinite we get that for all $y\in A^{*}$
$$
\left\| \left(\Sch_{G}(\St_{G}(y^{\omega}),\wt{Q}),\St_{G}(y^{\omega})\right)\right\|=\|\Gamma(G,\wt{Q},A^{\omega},y^{\omega})\|=\|\Gamma(S,Q,A^{\omega},y^{\omega})\|=\infty
$$
if and only if $\left[G:\St_{G}(y^{\omega})\right]$ is infinite, if and only if $y^{r}$ is not relation in $\mathcal{P}(\partial\mathrsfs{A})$ for any $r\ge 1$ .
\end{proof}
The following gap result holds.
\begin{theorem}\label{theo: burnside}
Let $\mathrsfs{A}=(Q,A,\cdot, \circ)$ be an $RI$-transducer such that $G=\mathcal{G}(\mathrsfs{A})$ is infinite. If $\St_{G}(\ul{u})=\id$ for all $\ul{u}\in A^{\omega}$, then $G$ is a Burnside group.
\end{theorem}
\begin{proof}
We have to prove that for all $g\in G$ there exists $n:=n(g)>0$ such that $g^n=\id$. Let $u\in A^{\ast}$. Since $\mathrsfs{A}$ is invertible there exists $k$ such that $g\cdot u^k=g$. This implies that $g^s\cdot u^k=g^s$ for all $s$. The automaton $\mathrsfs{A}$ is reversible and so there exists $n$ with the property $g^n\circ u^k=u^k$. These two properties together assure that
$$
g^n\circ (u^k)^{\omega}=g^n\circ u^{k\omega}= u^{k\omega}.
$$
So $g^n$ belongs to the stabilizer of $ u^{k\omega}$. By hypothesis $\St_G(u^{k\omega})=\id$ and this implies $g^n=\id$.
\end{proof}
The following corollary is an immediate consequence of Theorem \ref{theo: burnside} and Corollary \ref{cor: no positive dual}.
\begin{corollary}\label{cor: no trivial stabilizers RI not bireversible}
Let $\mathrsfs{A}$ be an $RI$-transducer such that $G=\mathcal{G}(\mathrsfs{A})$ is infinite. Then, just one among $\mathcal{G}(\mathrsfs{A})$ and $\mathcal{G}(\partial\mathrsfs{A})$ can have all trivial stabilizers in the boundary.
\end{corollary}
Note that by Theorem \ref{theo: burnside} and the non-existence of infinite Burnside groups on reversible $3$-states \cite{KiPiSa14} there are no $3$-state $RI$-transducers defining a group with all trivial stabilizers. In the bireversible case we obtain this stronger version of Corollary \ref{cor: no positive dual}. In this direction we have the following result.
\begin{corollary}
Let $\mathrsfs{A}$ be an $RI$-transducer that does not contain a connected component which is not bireversible. Then, $\mathcal{G}(\mathrsfs{A})$ can not have all trivial stabilizers in the border.
\end{corollary}
\begin{proof}
Assume, contrary to our claim, that $\mathrsfs{A}$ contains a connected component $\mathrsfs{R}$ that is an $RI$-transducer which is not bireversible such that $\mathcal{G}(\mathrsfs{A})$ has all trivial stabilizers in the boundary. Note that the same holds in particular for the subgroup $R=\mathcal{G}(\mathrsfs{R})$. Hence, by Corollary \ref{cor: not containing bireversible} $R$ is infinite because $\mathcal{P}(\mathrsfs{R})=\emptyset$. However, by Theorem \ref{theo: burnside} $R$ is a Burnside group, a contradiction.
\end{proof}
The previous corollary rules out the $RI$-transducers which are not bireversible as possible candidates for non-virtually nilpotent groups having all trivial stabilizers in the boundary. The following corollary is the best that we can obtain for the bireversible case.
\begin{corollary}
Let  $\mathrsfs{A}=(Q,A,\cdot, \circ)$ be a bireversible transducer such that $G=\mathcal{G}(\mathrsfs{A})$ is infinite. If $\St_{G}(\ul{u})=\id$ for all the essentially non-trivial points $\ul{u}\in \wt{A}^{\omega}$, then $\mathcal{G}(\partial\mathrsfs{A})$ is free of rank $|A|$.
\end{corollary}
\begin{proof}
By Theorem \ref{theo: enriched is equal} $\mathcal{G}(\mathrsfs{A})\simeq \mathcal{G}(\mathrsfs{A}^{-})$ is acting on $\wt{A}^{\omega}$. Since $G$ is infinite and $\St_{G}(\ul{u})$ is trivial for all the essentially non-trivial points, then all the Schreier automata
$$
\left(\Sch_{G}(\St_{G}(\ul{u}),\wt{Q}),\St_{G}(\ul{u})\right)
$$
are infinite for all the essentially non-trivial elements $\ul{u}\in \wt{A}^{\omega}$. Hence, by Theorem \ref{theo: finiteness schreier} $\mathcal{G}(\partial\mathrsfs{A})$ is free of rank $|A|$.
\end{proof}

\subsection{Some algorithmic consequences}
Very few is known from the algorithmic point of view of automata (semi)groups. Besides the word problem which is decidable, it seems that many problems for automata (semi)groups may be undecidable. In this direction it is worth mentioning here the result obtained by Gillibert on the undecidibility for the finiteness in semigroups automata \cite{Gilbert13}, the undecidability of the conjugacy problem \cite{VeSu11}, and the decidability of this problem in the contracting case \cite{BoBoSi13}.

In particular, checking the freeness of an automata group is still open. More formally, it is not known whether or not the following problem:
\\
\texttt{not-FREENESS}:
\begin{itemize}
\item Input: An invertible transducer $\mathrsfs{A}$.
\item output: Is $\mathcal{G}(\mathrsfs{A})$ not free?
\end{itemize}
is decidable, not even if we restrict it to the class of bireversible invertible transducers (\texttt{BIR-not-FREENESS}).
In this section we apply the results of the previous sections to obtain some equivalences from the algorithmic point of view. In particular, we reveal a strict connection between \texttt{FREENESS} (\texttt{BIR-FREENESS}) and checking the finiteness of the orbital graphs (Schreier graphs) in the border. For this reason we state the following two algorithmic problems regarding the dynamics of automata (semi)group: the first one is in the class of inverse transducer, and the second in the class of bireversible transducers.
\\
\texttt{FINITE-PERIODIC-ORBITAL}:
\begin{itemize}
\item Input: An inverse transducer $\mathrsfs{A}=(Q,\wt{A},\cdot,\circ)$.
\item output: Is there a finite orbital graph in the boundary centered on a periodic point $y^{\omega}$ with $\oo{y}\neq 1$?
\end{itemize}
By Theorem \ref{theo: finiteness essentially trivial}, in the class of bireversible inverse transducers, Schreier graphs of essentially trivial elements are finite. Hence, it makes sense to just look for finite Schreier graphs of essentially non-trivial elements.
\\
\\
\texttt{BIR-FINITENESS}:
\begin{itemize}
\item Input: A bireversible inverse transducer $\mathrsfs{A}=(Q,\wt{A},\cdot,\circ)$.
\item output: Is there a finite orbital graph in the boundary centered on an essentially non-trivial element?
\end{itemize}
The following results are consequences of Theorem \ref{theo: finiteness schreier} and Theorem \ref{theo: characterization free} and the fact that passing from a transducer to its dual is an $\mathcal{O}(\log n)$-space bounded reduction (see for instance \cite{papa}).
\begin{theorem}
\texttt{not-FREENESS} is decidable if and only if \texttt{FINITE-PERIODIC-ORBITAL} is decidable. Furthermore, if they are decidable, then they both belong to the same computational class.
\end{theorem}
\begin{theorem}
\texttt{BIR-not-FREENESS} is decidable if and only if \texttt{BIR-FINITENESS} is decidable. Furthermore, if they are decidable, then they both belong to the same computational class.
\end{theorem}

\section{Transducers with sink-state}\label{sec: transducers with sink}
We address the question whether or not a free group of rank grater that one can be generated by a transducer containing a sink-state. This problem is strictly related to the existence of certain words which we call \emph{fragile}. We consider the rather broad class $\mathcal{S}_{a}$ of invertible transducers $\mathrsfs{A}=(Q,A,\cdot,\circ)$ for which the sink-state $e\in Q$ is accessible from every state, we also make the extra assumption that the sink acts like the identity. We recall that a sink state $e$ is such that $e\cdot v=e$ and $e\circ v=v$, for any $v\in A^{\ast}$. Note that these conditions are equivalent to have transitions $a\longfr{e}{e}a$ for any $a\in A$ in $\partial\mathrsfs{A}$. Let $g\in\mathcal{G}(\mathrsfs{A})$. We recall that an element $\underline{w}\in A^{\omega}$ is $g$-regular, if there exists a prefix $\underline{w}[n]$ of $\underline{w}$ such that $g\cdot \underline{w}[n]=e$ (see, for instance \cite{Nekra10}). If for every $n\geq 1$, $g\cdot \underline{w}[n]\neq e$ then $\underline{w}$ is said to be $g$-singular. Given an automaton $\mathrsfs{A}=(Q,A,\cdot,\circ)$, we can define its \textit{reduction automaton} $\mathcal{R}(\mathrsfs{A})=(\mathcal{R}(Q),A,\widehat{\cdot},\widehat{\circ})$ as the smallest automaton where we identify with the sink state all the maximal connected components $Q_1,\ldots, Q_t\subseteq Q$ such that for every $v\in A^{\ast}$ and $q\in Q_i$, $q\cdot v\in Q_i $ and $q\circ v=v$. Therefore, this implies $q=\id$ in $\mathcal{G}(\mathrsfs{A})$. Denote by $\overline{Q}$ the set $Q_1\cup\ldots\cup Q_t$. If $\overline{Q}\neq \emptyset$ we put $\mathcal{R}(Q)=(Q\setminus \overline{Q})\sqcup \{e\}$, where $e$ is a sink state, and for any $p\in Q\setminus \overline{Q}$ such that if $p\cdot u=q\in \overline{Q}$ we define $p \ \widehat{\cdot} \ u=e$ and $p \ \widehat{\circ} \ u=p\circ u$, and if $p\cdot u=q\notin \overline{Q}$ then $p \ \widehat{\cdot} \ u=p\cdot u$ and $p \ \widehat{\circ} \ u=p\circ u$. Otherwise, if $\overline{Q}= \emptyset$ then $\mathcal{R}(Q)=Q$ and $\mathcal{R}(\mathrsfs{A})=\mathrsfs{A}$. Notice that if $\mathcal{R}(\mathrsfs{A})$ contains a sink-state, this sink is unique and coincides with $e$. Moreover, it is straightforward to check that $\mathcal{G}(\mathrsfs{A})\simeq \mathcal{G}(\mathcal{R}(\mathrsfs{A}))$ holds.

With the notion of $g$-regular ($g$-singular) elements the class $\mathcal{S}_{a}$ has the following property

\begin{proposition}
If $\mathrsfs{A}\in \mathcal{S}_{a}$ with $\mathrsfs{A}=(Q,A,\cdot,\circ)$, then every element $g\in\mathcal{G}(\mathrsfs{A})$ has a $g$-regular element in $A^{\omega}$. Vice versa if the action of every element $g\in\mathcal{G}(\mathrsfs{A})$ has a $g$-regular element in $A^{\omega}$ then $\mathcal{R}(\mathrsfs{A})\in \mathcal{S}_{a}$.
\end{proposition}
\begin{proof}
The proof is by induction on the length $m$ of an element $q_{1}\ldots q_{m}\in \wt{Q}^{*}$, representing some element $g\in \mathcal{G}(\mathrsfs{A})$ via the canonical map $\wt{Q}^{*}\rightarrow \mathcal{G}(\mathrsfs{A})$. We consider two following cases.
\begin{itemize}
\item If $q_{m}\in Q$, then since $\mathrsfs{A}\in \mathcal{S}_{a}$ there is a $u\in A^{*}$ such that $q_{m}\cdot u=e$. Hence we get
$q_{1}\ldots q_{m}\cdot u=q_{1}'\ldots q_{m-1}'e$. By the induction hypothesis there is a $w\in A^{*}$ such that $q_{1}'\ldots q_{m-1}'\cdot w=e^{m-1}$, whence $q_{1}\ldots q_{m}\cdot uw=e^{m}$, and so $uwA^{\omega}$ is a set of $g$-regular elements.
\item If $q_{m}\in Q^{-1}$, then there is a $u\in A^{*}$ such that $q_{m}^{-1}\cdot u=e$. Since $\mathrsfs{A}$ is invertible consider $u'=q_{m}\circ u$. Thus, $q_{m}\cdot u'=e$ and the rest of the proof proceeds like in the previous case.
\end{itemize}

Vice versa, let us show that if $\mathcal{R}(\mathrsfs{A})\not\in  \mathcal{S}_{a}$, then there is $g\in\mathcal{G}(\mathrsfs{A})$ that does not admit any $g$-regular sequence. The first condition implies that either $\mathcal{R}(\mathrsfs{A})$ has a unique sink, but the sink is not accesible from every state or $\mathcal{R}(\mathrsfs{A})=\mathrsfs{A}$ has no sink. The latter case gives that for every $g\in \mathcal{G}(\mathrsfs{A})$, $g$ has no $g$-regular element. The former case implies that there exists a state $q\in \mathcal{R}(Q)$ with the property that $q \ \widehat{\cdot} \ w\neq e$ for every $w\in A^{\ast}$. In particular $q$ does not admit any $q$-regular sequence.
\end{proof}
For every $q\in Q$, the erasing morphism $\epsilon_{q}:\wt{Q}^{*}\rightarrow \wt{Q}^{*}$ is the homomorphism that sends $q, q^{-1}$ to the empty word $1$. In this context, a word $w\in\wt{Q}^{*}$ is called trivial if $\oo{\epsilon_{e}(w)}=1$. The \emph{content} of $w\in\wt{Q}^{*}$ is the smallest subset $Q'\subseteq Q$ such that $w\in\wt{Q'}^{*}$. In other words, the content $C(w)$ is the set of elements $q\in Q$ such that either $q$ or $q^{-1}$ occurs in $w$; we say that $w$ has \emph{$m$ occurrences} if $|C(w)|=m$.
\begin{definition}[Fragile words]
Let $\mathrsfs{A}=(Q,A,\cdot,\circ)$ be an invertible transducer in $\mathcal{S}_{a}$ with sink state $e$. A non-trivial word $w\in\wt{Q}^{*}$ is called fragile if there is $a\in A$ such that $w\cdot a$ is trivial.
\end{definition}

We get the following necessary condition for $\mathcal{G}(\mathrsfs{A})$ to be non-free.

\begin{proposition}
With the above conditions, if $\mathcal{G}(\mathrsfs{A})$ is not free, then there is a shortest non-trivial relation $w\in\wt{Q}^{*}$ which is a fragile word.
\end{proposition}
\begin{proof}
Take any shortest non-trivial relation $w\in\wt{Q}^{*}$, and assume $w=q_{k}^{f_{k}}\ldots q_{1}^{f_{1}}$, for some $f_{i}\in\{1,-1\}$. Since $q_{1}$ ($q_{1}^{-1}$ in $\mathrsfs{A}^{-1}$) is connected to the sink-state $e$ ($e^{-1}$) by some suitable word $u_{1}\ldots u_{\ell}=u\in A^{*}$, we get $w\cdot u=z_{k}\ldots z_{2}e^{f_{1}}$. Therefore, since $z_{k}\ldots z_{2}e^{f_{1}}$ is still a relation by Theorem \ref{theo: charact relations} and $|\epsilon_{e}(w\cdot u)|<k$, by minimality, we get that $z_{k}\ldots z_{2}e^{f_{1}}$ is trivial. Hence, $w\cdot (u_{1}\ldots u_{\ell-1})$ is fragile.
\end{proof}
Therefore, it is evident how fragile words play an important role as relations occurring in non-free groups defined by transducers with a sink state.

\subsection{A particular series of automata with sink state}\label{sec: with sink}
In this section we present a series of automata with sink such that, combined with already known transducers generating free groups, generate transducers with a sink state defining a free group. We also discuss about the sets of relations defining the groups generated by the duals of these transducers.
\begin{definition}
Given two transducers $\mathrsfs{A}=(Q,A,\cdot,\circ), \mathrsfs{B}$ on the same set of states $Q$, we say that $\mathrsfs{B}$ dually embeds into $\mathrsfs{A}$, in symbols $\mathrsfs{B}\hookrightarrow_{d} \mathrsfs{A}$, if $\partial \mathrsfs{B}$ is a proper connected component of $\partial \mathrsfs{A}$.
\end{definition}
Note that with the above condition if $B$ is the alphabets of $ \mathrsfs{B}$, then the coupled actions of $ \mathrsfs{B}$ are simply the coupled-actions of $ \mathrsfs{A}$ restricted to the alphabet $B$. For this reason, without loss of generality, we can write $ \mathrsfs{B}=(Q,B,\cdot,\circ)$.
\begin{lemma}\label{lem: dual embed}
Given two invertible transducers $\mathrsfs{A}=(Q,A,\cdot,\circ)$, $\mathrsfs{B}=(Q,B,\cdot,\circ)$ such that $\mathrsfs{B}\hookrightarrow_{d} \mathrsfs{A}$ there is an epimorphism $\psi\colon \mathcal{G}(\mathrsfs{A})\twoheadrightarrow  \mathcal{G}(\mathrsfs{B})$.
\end{lemma}
\begin{proof}
Let $\mathcal{G}(\mathrsfs{A})=F_{Q}/N$, $\mathcal{G}(\mathrsfs{B})=F_{Q}/M$, and
$$
\mathcal{N}\subseteq \bigcap_{a\in A}L\left((\partial\mathrsfs{A})^{-},a\right), \quad \mathcal{M}\subseteq \bigcap_{b\in B}L\left((\partial\mathrsfs{B})^{-},b\right)
$$
be the sets as in Theorem \ref{theo: charact relations}. The fact that $\partial \mathrsfs{B}$ is a proper connected component of $\partial \mathrsfs{A}$ implies that $\mathcal{N}\subseteq \mathcal{M}$. Hence, $N\le M$ from which the statement follows.
\end{proof}
\noindent Note that an analogous lemma for semigroup automata holds. \\
The\emph{ sink transducer on an alphabet $A$} is $\mathrsfs{E}=(\{e\},A, \cdot, \circ)$ with $e\cdot a=e$, $e\circ a=a$, for all $a\in A$. Given a transducer $\mathrsfs{A}=(Q,A,\cdot,\circ)$ such that $e\notin Q$, we denote by $\mathrsfs{A}^{e}=\mathrsfs{A}\sqcup \mathrsfs{E}$. Note that adding a sink state does not change the group, i.e. $\mathcal{G}(\mathrsfs{A}^{e})\simeq \mathcal{G}(\mathrsfs{A})$. We now introduce a class of auxiliary automata with some interesting features. We present then via their duals, let $Q$ be a finite set, then $\partial\mathrsfs{S}_{Q}=(Q, Q\cup\{e\}, \circ, \cdot)$, where the actions (see Fig. \ref{Fig: automata with sink free}) are defined by:
$$
q\circ x=q,\quad
q\cdot x=\begin{cases}
e, \mbox{ if } x=q\\
x, \mbox{ otherwise }
\end{cases}
$$
for all $q\in Q$.
\begin{figure}
		\begin{center}
			\begin{tikzpicture}[>=latex, shorten >=1pt, shorten <=1pt]
				\tikzstyle{normal_node}= [draw,circle,inner sep=0pt,thick,minimum size=1cm]
				\draw (0,0) node [normal_node] (0) {$q_{1}$};
				\draw (3,0) node [normal_node] (1) {$q_{2}$};
				\draw (5,0) node [draw=none] (2) {$\ldots$};
				\draw (7,0) node [normal_node] (3) {$q_{k}$};				
				\path (0)
				(0) edge[->, out=110, in=70, distance=1cm, thick] node[above] {$q_{1}|e$} (0)
				(0) edge[->, out=290, in=250, distance=1cm, thick] node[below] {$\{y|y\colon y\neq q_{1}\}$} (0)
				(1) edge[->, out=110, in=70, distance=1cm, thick] node[above] {$q_{2}|e$} (1)
				(1) edge[->, out=290, in=250, distance=1cm, thick] node[below] {$\{y|y\colon y\neq q_{2}\}$} (1)
				(3) edge[->, out=110, in=70, distance=1cm, thick] node[above] {$q_{k}|e$} (3)
				(3) edge[->, out=290, in=250, distance=1cm, thick] node[below] {$\{y|y\colon y\neq q_{k}\}$} (3);			
			\end{tikzpicture}
			\label{automaton I}
		\end{center}
		\caption{The transducer $\partial\mathrsfs{S}_{Q}$ with $Q=\{q_{1},\ldots, q_{k}\}$.} \label{Fig: automata with sink free}
\end{figure}
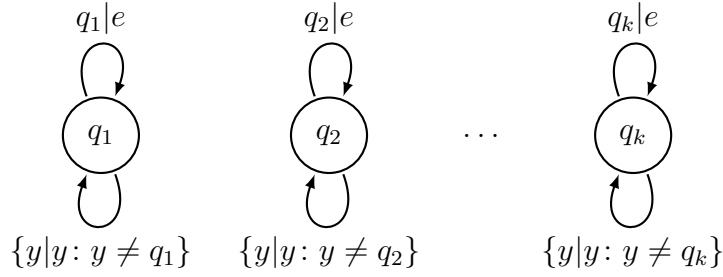
Note that with these actions $\partial\mathrsfs{S}_{Q}$ is reversible, hence by Proposition \ref{prop: dual prop} $\mathrsfs{S}_{Q}$ is invertible. The following proposition shows an important characteristic of these automata.
\begin{proposition}
Let $\mathrsfs{B}=(Q,B,\cdot,\circ)$ be an invertible transducer. Let $H\subseteq Q$, and put
$$
\partial\mathrsfs{A}_{H}=\partial\mathrsfs{B}^{e}\sqcup \partial \mathrsfs{S}_{H}
$$
Then, $\mathrsfs{A}_{H}$ is an automaton with sink and $\mathcal{G}(\mathrsfs{B})$ is a quotient of $\mathcal{G}(\mathrsfs{A}_{H})$. In particular, $\mathrsfs{A}_{Q}\in\mathcal{S}_{a}$ the class of transducers with sink that is accessible from any state.
\end{proposition}
\begin{proof}
This is a consequence of Lemma \ref{lem: dual embed} and the definitions.
\end{proof}
As a consequence we get the following immediate corollary.
\begin{corollary}
For any invertible transducer $\mathrsfs{B}$ defining a free group $F_{m}$, the transducer $\mathrsfs{A}_{Q}$ defined by
$$
\partial\mathrsfs{A}_{Q}=\partial\mathrsfs{B}^{e}\sqcup \partial \mathrsfs{S}_{Q}
$$
belongs to the class $\mathcal{S}_{a}$, and it defines a free group $F_{m}$.
\end{corollary}
\noindent In particular taking any transducer in \cite{StVoVo2011} one is able to explicitly exhibit such machines.
In general, if $\mathrsfs{S}_{Q}$ dually embeds into any other transducer, then its minimal defining relations are all fragile words, and in this simple case these words have a nice property which is independent from the structure of the transducer and it can be described in a purely combinatorial way. A word $w\in \wt{Q'}^{*}$, with content $Q'\subseteq Q$, is called \emph{strongly fragile} on the set $Q'$ if
$$
\oo{\epsilon_{q}(w)}=1,\;\forall q\in Q'
$$
Although these words have this simple description, it seems that they are difficult to characterize. However, there is a simple way to generate some of them. If we order arbitrarily the set $Q'=\{q_1, \ldots, q_m\}$, where $m\geq 2$, we may define the set $\mathcal{C}(Q')$ of \textit{commutator words on $Q'$} given by the following properties. Let $\mathcal{C}_1$ be the set of words of type
$$
[q_i^{e_i},q_j^{e_j}]:=q_i^{e_i}q_j^{e_j}q_i^{-e_i}q_j^{-e_j}  \, \, \forall i\neq j, \, \, e_i,e_j\in \mathbb{Z}\setminus\{0\}
$$
and, inductively, let $\mathcal{C}_i$ be the set of words $[q_i^{e_i},v^{e}]:= q_i^{e_i}v^{e}q_i^{-e_i}v^{-e}$, $e_i,e\in \mathbb{Z}\setminus\{0\}$, where $q_i$ is not an occurrence of $v$ and $v\in \mathcal{C}_{i-1}$. We put $\mathcal{C}(Q'):=\mathcal{C}_{m-1}$.
This construction generates some strongly fragile words as the following proposition shows.
\begin{proposition}\label{lem: strongly fragile}
The following facts hold:
\begin{enumerate}
  \item[i)] The set $\mathcal{C}(Q')$ is a subset of the set of strongly fragile words on $Q'\subseteq Q$.
  \item[ii)] If $|Q'|=m$, the shortest strongly fragile words on $Q'$ which are in set $\mathcal{C}(Q')$ have length $3(2^{m-1})-2$.
\end{enumerate}
\end{proposition}
\begin{proof} $i)$. Let us prove the statement by induction on $|Q'|=m\geq 2$. For $m=2$, we have that $\mathcal{C}(Q')$ is the set of commutator words in $q_1$ and $q_2$. It is straightforward to check that, for every $w\in \mathcal{C}(Q')$ one has $\overline{\epsilon_{q_1}(w)}=\overline{\epsilon_{q_2}(w)}=1$.
Suppose now that $1.$ is true for $|Q'|=m-1$. Let $w$ be a commutator word on $m$ occurrences, $w\in \mathcal{C}(Q'):=\mathcal{C}_{m-1}$. By definition $w=[q_i^{e_i},v^{e}]$, for some $v\in \mathcal{C}_{m-2}$ that does not contain $q_i$ as occurrence. Notice that $v$ is a commutator word on $Q' \setminus\{q_i\}$. Then $\overline{\epsilon_{q_i}(w)}=\overline{v^ev^{-e}}=1$ and for $j\neq i$, one gets $\overline{\epsilon_{q_j}(w)}= \overline{q_i^{e_i}q_i^{-e_i}}=1$, since $\overline{\epsilon_{q_j}(v^e)}=\overline{\epsilon_{q_j}(v^{-e})}=1$.\\
$ii).$ Clearly the shortest words are obtained when the exponents $e_i$'s belong to $\{-1, +1\}$. For $m=2$ we get $4=3(2^{2-1})-2$. By induction we have
$$
|w|=4+ 2|v|=2+2(3(2^{m-2})-2)=3(2^{m-1})-2
$$
\end{proof}
\noindent In what follows, we show that the commutator words are not the only strongly fragile words on $Q'$. On the other hand we conjecture that the such words are the shortest ones. This would give an exponential (on the number of occurrences) lower
bound on the length of the shortest relation on a group defined by a transducer in which $\mathrsfs{S}_{Q}$ dually embeds. This is may be interesting in understanding the existence of upper bounds that ensure the existence of relations. Probably, this bound does not exist (so this may be a clue of the undecidability of \texttt{not-FREENESS}), and a large (exponential) bound can be a sign of this fact.

\begin{figure}
\begin{center}
\begin{tikzpicture}
	\tikzstyle{axis}   = [-latex,black!100];

	\coordinate (P1) at (-7cm,1.5cm);
	\coordinate (P2) at (8cm,1.5cm);

	\coordinate (A1) at (0em,0cm); 
	\coordinate (A2) at (0em,-2cm); 

	\coordinate (A3) at ($(P1)!.8!(A2)$); 
	\coordinate (A4) at ($(P1)!.8!(A1)$);

	\coordinate (A7) at ($(P2)!.7!(A2)$);
	\coordinate (A8) at ($(P2)!.7!(A1)$);

	\coordinate (A5) at
	  (intersection cs: first line={(A8) -- (P1)},
			    second line={(A4) -- (P2)});
	\coordinate (A6) at
	  (intersection cs: first line={(A7) -- (P1)},
			    second line={(A3) -- (P2)});

	\fill[gray!90] (A2) -- (A3) -- (A6) -- (A7) -- cycle;
	
	\fill[gray!50] (A3) -- (A4) -- (A5) -- (A6) -- cycle;

	\fill[gray!30] (A5) -- (A6) -- (A7) -- (A8) -- cycle;

	\draw[thick,dashed] (A5) -- (A6);
	\draw[thick,dashed] (A3) -- (A6);
	\draw[thick,dashed] (A7) -- (A6);
	
	\fill[gray!50,opacity=0.2] (A1) -- (A2) -- (A3) -- (A4) -- cycle;
	\fill[gray!90,opacity=0.2] (A1) -- (A4) -- (A5) -- (A8) -- cycle;

	\draw[thick] (A1) -- node[right]{$b$}(A2);
	\draw[thick] (A3) -- node[left]{$b$}(A4);
	\draw[thick] (A7) -- node[right]{$b$}(A8);
	\draw[thick] (A1) -- node[below, pos=0.6]{$a$}(A4);
	\draw[thick] (A1) -- node[below right]{$c$} (A8);
	\draw[thick] (A2) -- node[below left]{$a$}(A3);
	\draw[thick] (A2) -- node[below right]{$c$} (A7);
	\draw[thick] (A4) -- node[above left]{$c$}  (A5);
	\draw[thick] (A8) -- node[above right]{$a$}(A5);
	
	\foreach \i in {1,...,8}
	{
	  \draw[fill=black] (A\i) circle (0.15em);
	}
	\draw[fill=black] (A4) circle (0.15em) node[above left] {$r$};
\end{tikzpicture}
\end{center}
\caption{The rooted 2-simplicial complex $(\mathcal{Q},r)$.}\label{Fig cube}
\end{figure}
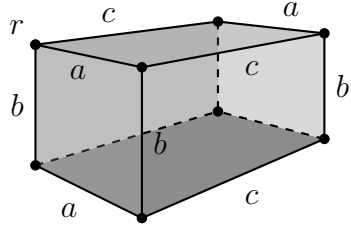

We now describe a geometric way to build some examples, however we do not know a characterization of all the strongly fragile words. For simplicity of exposition we restrict ourself to the alphabet $A=\{a,b,c\}$, but the following construction can be generalized to larger alphabets. Consider the rooted 2-simplicial complex $(\mathcal{Q},r)$ such that the geometrical realization of the 2-simplicial complex $\mathcal{Q}$ is the unit cube in $\mathbb{R}^{3}$ with orthonomal bases $\{\mathbf{i}_{x}, x\in A\}$, and let us fix a root (base-point) $r$ among the 0-faces of $\mathcal{Q}^{0}$. Let us label the 1-faces $\mathcal{Q}^{1}$ in such a way that the 1-faces along $\mathbf{i}_{x}$ are labelled by $x$, $x\in A$ (see Fig. \ref{Fig cube}). Consider now the set $\Gamma_{r}$ of closed paths in the $1$-simplicial complex $\mathcal{Q}^{1}$ starting from $r$; an element $p\in \Gamma_{r}$ can be represented as
$$
r=v_{0}\mapright{e_{1}}v_{1}\mapright{e_{2}}\ldots v_{i}\mapright{e_{i+1}}v_{i+1}\ldots v_{m-1}\mapright{e_{m}}v_{m}=r
$$
where $v_{i} \in  \mathcal{Q}^{0}$, and $e_{i} \in  \mathcal{Q}^{1}$ such that $v_{i}, v_{i+1}$ are the $0$-faces of $e_{i+1}$; in case
the only repeated vertex is $v_{0}$ we speak of a cycle. We can define a labeling map $\mu:  \Gamma_{r}\rightarrow \wt{A}^{*}$ defined by $\mu(p)=\mu(e_{1})\ldots \mu(e_{m})$ where  $\mu(e_{i})=x$ if $v_{i+1}=v_{i}+\mathbf{i}_{x}$, otherwise if $v_{i}=v_{i+1}+\mathbf{i}_{x}$ we put $\mu(e_{i})=x^{-1}$, $x\in A$. For each face $\sigma\in  \mathcal{Q}^{2}$ contained in an affine plane perpendicular to $\mathbf{i}_{c}$, and for a cycle $v_{0}\Pth{u}v_{0}$ in $\sigma^{1}$ we have that either $u$ is a cyclic shift of $aba^{-1}b^{-1}$ or $bab^{-1}a^{-1}$; with respect to the orientation $\mathbf{i}_{c}$, this corresponds to a clockwise, or a counter-clockwise travel in $\sigma^{1}$ around $\mathbf{i}_{c}$. Hence, we can define a map $\phi$ from the set $\Lambda_{r}$ of closed paths $p\in \Gamma_{r}$ of the form
$$
p=r\Pth{s}v_{0}\Pth{u}v_{0}\Pth{s^{-1}}r
$$
for any cycle $v_{0}\Pth{u}v_{0}$ in $\sigma^{1}$, $\sigma\in \mathcal{Q}^{2}$, into $\mathbb{R}^{3}$ sending $p$ to either $\mathbf{i}_{c}$ or $-\mathbf{i}_{c}$ whether or not $v_{0}\Pth{u}v_{0}$ is a cycle that is traveling clockwise around the normal $\mathbf{i}_{c}$ to the face $\sigma$. The set $\Lambda_{r}$ can be extended to a submonoid $\Lambda_{r}^{*}$ of $\Gamma_{r}$ by concatenation of closed paths. Consequently, the map $\phi$ can be naturally extended via $\phi(p_{1}p_{2})=\phi(p_{1})+\phi(p_{2})$. The following proposition shows another way to generate strongly fragile words.
\begin{proposition}
Every non-reduced word $\mu(p)\in \wt{A}^{*}$ such that $p\in \Lambda_{r}^{*}$ satisfies $\phi(p)=0$ is strongly fragile.
\end{proposition}
\begin{proof}
We just give a sketch of the proof and we leave the details to the reader. Take any $\mu(p)\in \wt{A}^{*}$ such that $p\in \Lambda_{r}^{*}$ as in the statement. Note that, for any $p_{1}\in\Lambda_{r}$ such that $\phi(p_{1})\in\{\mathbf{i}_{c},-\mathbf{i}_{c}\}$, $c\in A$, we get $\oo{\epsilon_{x}(p_{1})}=1$ for all $x\in A\setminus\{c\}$. Hence, for any $c\in A$, it is not difficult to check that
$$
\oo{\epsilon_{c}(\mu(p))}=\oo{\mu\left(\epsilon_{c}\left(q_{1}\ldots q_{k}\right)\right)}
$$
for some $q_{1},\ldots, q_{k}\in \Lambda_{r}$ with $\phi(q_{i})\in \{\mathbf{i}_{c},-\mathbf{i}_{c}\}$, $i=1,\ldots, k$, and $\sum_{i=1}^{k}\phi(q_{i})=0$. Now, it is not difficult to check that if $q_{i}, q_{j}\in\Lambda_{r}$ have the property that $\phi(q_{i})=-\phi(q_{j})$, then $\oo{\epsilon_{c}\left(\mu(q_{1})\right)}=\oo{\epsilon_{c}\left(\mu(q_{2})\right)}^{-1}$. From this fact and $\sum_{i=1}^{k}\phi(q_{i})=0$ we obtain $\oo{\epsilon_{c}(\mu(p))}=\oo{\mu\left(\epsilon_{c}\left(q_{1}\ldots q_{k}\right)\right)}=1$, and this concludes the proof.
\end{proof}
\noindent For instance, we can immediately compute the following strongly fragile word:
\begin{align}
\nonumber &&w=\oo{(ab^{-1}cbc^{-1}a^{-1})(ab^{-1}a^{-1}b)(cb^{-1}aba^{-1}c^{-1})(cb^{-1}c^{-1}b)}=\\
\nonumber &&=ab^{-1}cbc^{-1}b^{-1}a^{-1}bcb^{-1}aba^{-1}b^{-1}c^{-1}b
\end{align}
We leave open the problem of characterizing strongly fragile words, maybe, generalizing the previous geometric construction it is possible to find a way to characterize them.

\subsection{Cayley type transducers}\label{section: cayley}
We start with some examples considering Cayley machines types of transducers. The idea is to color a Cayley automaton in such a way that one obtains a reversible transducer which is interpreted as the dual transducer of an invertible one generating a group. This dual transducer generates in general a semigroup, unless it is bireversible. This coloring approach can be found also in \cite{DaRo2013}. The coloring that we present is in some way the ``easiest'', since besides the transitions entering and exiting from the  identity state, the others act like the identity.

Since both the states and the alphabet are elements of a group we need to carefully fix the notation. Let $G$ be a finite group with neutral element $e$. For any $g\in G$ we denote by $g^{\m}$ the unique inverse of $g$ in $G$. As usual, the product of $n$ elements $g_1,\ldots, g_n$ in $G$ is $g_1g_2\cdots g_n$. When we are interested in strings of elements in $G$, without invoking the composition law of the group $G$, we ``parenthesize'' the elements of $G$ and we put $(G)=\{(g):g\in G\}$, hence an element of $(G)^{*}$ is of the form $(g_1)(g_2)\cdots (g_n)$, where $g_i\in G$.

\begin{figure}
		\begin{center}
			\begin{tikzpicture}[>=latex, shorten >=1pt, shorten <=1pt]
				\tikzstyle{normal_node}= [draw,circle,inner sep=0pt,thick,minimum size=0.8cm]
				\draw (0,0) node [normal_node] (1) {$\textbf{1}$};
				\draw (4,-2) node [normal_node] (0) {$\textbf{0}$};
				\draw (0,-4) node [normal_node] (2) {$\textbf{2}$};
				\draw (6,0) node [normal_node] (0') {$0$};
				\draw (9,0) node [normal_node] (1') {$1$};
				\draw (6,-4) node [normal_node] (2') {$2$};		
				\path (0)
				(1) edge[->, thick] node[left] {$1|1$} (2)
				(2) edge[->, thick] node[below] {$1|0$} (0)
				(0) edge[->, thick] node[above] {$1|1$} (1)
				(2) edge[->, bend left=40, thick] node[left] {$2|2$} (1)
				(1) edge[->, bend left=40, thick] node[above] {$2|0$} (0)
				(0) edge[->, bend left=40, thick] node[below] {$2|2$} (2)
				(0) edge[->, out=-30, in=30, distance=1cm, thick] node[right] {$0|0$} (0)
				(1) edge[->, out=180, in=120, distance=1cm, thick] node[left] {$0|0$} (1)
				(2) edge[->, out=180, in=240, distance=1cm, thick] node[left] {$0|0$} (2)
				(1') edge[->, thick]  node[above] {$\textbf{2}|\textbf{0}$} (0')
				(2') edge[->, thick] node[right] {$\textbf{1}|\textbf{0}$} (0')
				(0') edge[->, out=180, in=120, distance=1cm, thick] node[inner sep=10pt, above] {$\left\{x|x: x=\textbf{0},\textbf{1},\textbf{2}\right\}$} (0')
				(2') edge[->, out=180, in=240, distance=1cm, thick] node[left] {$\textbf{2}|\textbf{1}$} (2')
				(2') edge[->, out=0, in=-60, distance=1cm, thick] node[right] {$\textbf{0}|\textbf{2}$} (2')
				(1') edge[->, out=0, in=60, distance=1cm, thick] node[right] {$\textbf{1}|\textbf{2}$} (1')
				(1') edge[->, out=0, in=-60, distance=1cm, thick] node[right] {$\textbf{0}|\textbf{1}$} (1');				
			\end{tikzpicture}
			\label{automaton I}
		\end{center}
		\caption{On the left the transducer $\mathrsfs{C}(\mathbb{Z}_3)$, on the right its dual $\partial\mathrsfs{C}(\mathbb{Z}_3)$.} \label{Fig: automata cayley 3}
\end{figure}
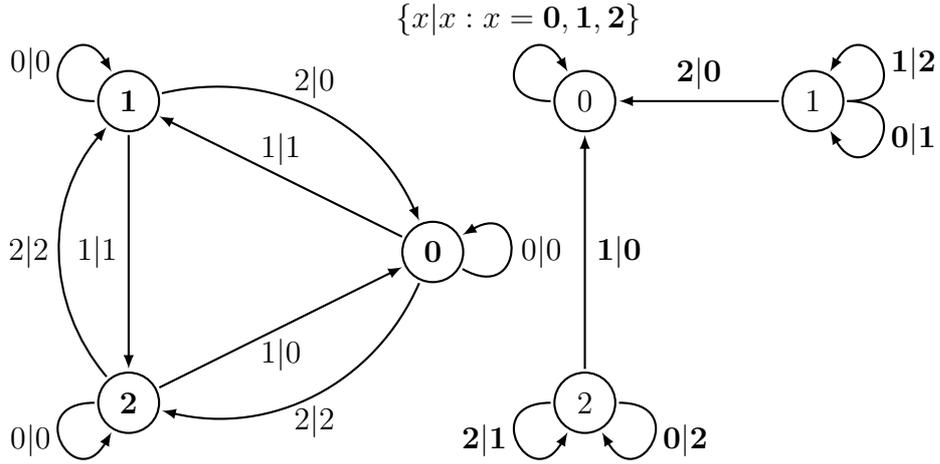
The \textit{0-transition Cayley machine} $\mathrsfs{C}(G)=(\mathbf{G},(G),\circ,\cdot)$ is the transducer defined on the alphabet $\mathbf{G}=\{\mathbf{g} : \ g\in G\}$ whose transitions are of the form
\begin{itemize}
\item $\mathbf{g}\vlongfr{(x)}{(x)}\mathbf{gx}$ for all $g,x\in G$ such that $g\neq x^{\m}$;
\item $\mathbf{g}\vlongfr{(x)}{(e)}\mathbf{e}$ for all $g,x\in G$ such that $g = x^{\m}$.
\end{itemize}
Recall that $\mathrsfs{C}(G)^{-}$ denotes the enriched transducer of $\mathrsfs{C}(G)$, and this acts on the rooted tree $\wt{(G)}^{*}$ where the set of formal inverses of $(G)$ is given by $(G)^{-1}=\{(g)^{-1}: \ (g)\in (G)\}$. With this notation, the inverse transitions of $\mathbf{g}\longfr{(x)}{(x)}\mathbf{gx}$, $\mathbf{g}\longfr{(x)}{(e)}\mathbf{e}$ are $\mathbf{gx}\vvlongfr{(x)^{-1}}{(x)^{-1}}\mathbf{g}$, $\mathbf{e}\vvvlongfr{(x)^{-1}}{(e)^{-1}}\mathbf{g}$, respectively.
\\
\\
Similarly, we define the \textit{bi-0-transition Cayley machine} $\mathrsfs{\wt{C}}(G)=(\mathbf{G},(G),\circ,\cdot)$ with transitions given by:\begin{itemize}
\item $\mathbf{g}\vlongfr{(x)}{(x)}\mathbf{gx}$ for all $g,x\in G$ such that $g\neq x^{\m}$ and $g\neq e$;
\item $\mathbf{g}\vlongfr{(x)}{(e)}\mathbf{e}$ for all $g,x\in G$ such that $g= x^{\m}$ and $g\neq e$;
\item $\mathbf{e}\vlongfr{(x)}{(e)}\mathbf{x}$ for all $x\in G$.
\end{itemize}
Notice that these machines act like the identity except for the transitions passing through the state $\mathbf{e}$.
\\
We stress once again the fact that $g$ represents an element of the finite group $G$ and $(g)$ an element of the alphabet $(G)$ of the Cayley machine.

\begin{lemma}\label{lem:prefragilita}
Let $\mathbf{h}$ be a state of $\mathrsfs{C}(G)=(\mathbf{G},(G),\circ,\cdot)$ or $\mathrsfs{\wt{C}}(G)=(\mathbf{G},(G),\circ,\cdot)$. Then $\mathbf{h}\circ (g_1)\cdots (g_k)=\mathbf{h}$ if and only if $g_1\cdots g_k=e$.
\end{lemma}
\begin{proof}
It is enough to observe that a closed path in $\mathrsfs{C}(G)=(\mathbf{G},(G),\circ,\cdot)$ (or $\mathrsfs{\wt{C}}(G)=(\mathbf{G},(G),\circ,\cdot)$) corresponds to a closed path in the Cayley graph of $G$ with respect to the generating set $G$. This implies that any closed path is a relation of $G$.
\end{proof}

\begin{lemma}\label{lem:fragilita}
Let $u=(u_1)(u_2)\cdots (u_n)$ a relation of minimal length in $\mathcal{G}(\partial \mathrsfs{C}(G))$ (or $\mathcal{G}(\partial \mathrsfs{\wt{C}}(G))$). Then there exists $\mathbf{h}\in \mathbf{G}$ such that $v:= \oo{\epsilon_{(e)}(\mathbf{h}\cdot u)}=1$.
\end{lemma}
\begin{proof}
We study only the case of $\mathcal{G}(\partial \mathrsfs{C}(G))$, the other case of $\mathcal{G}(\partial \mathrsfs{\wt{C}}(G))$ is left to the reader. We have two possibilities:
\begin{itemize}
\item The element $(u_1)\in (G)$. Then $\mathbf{u_1}^{\m}\cdot (u_1)=(e)$. This implies that
$$
|v|=|\oo{\epsilon_{(e)}(\mathbf{h}\cdot u)}|\leq |u|-1
$$
Since the set of relations is closed under the operation $\cdot$ we have that $v$ is also a relation. The minimality of the hypothesis implies that $v$ can be reduced to the trivial element $(e)$.
\item If $(u_1)^{-1}\in (G)^{-1}$ it is enough to take $\mathbf{h}=\mathbf{e}$.
\end{itemize}
\end{proof}

Now we give the self-similar presentation of the elements of $\mathcal{G}(\partial \mathrsfs{C}(G))$ and $\mathcal{G}(\partial \wt{\mathrsfs{C}}(G))$. Since the group $G$ is supposed to be finite we can order its elements as $G=\{g_0=e,g_1,\ldots, g_{|G|}\}$. The group $\mathcal{G}(\partial \mathrsfs{C}(G))$ acts on the set $\mathbf{G}^{\ast}\sqcup \mathbf{G}^{\omega}$ and it is generated by $(G)$. In a natural way, any $(g_i)$ induces a permutation $\sigma_i\in Sym(|\mathbf{G}|)$ on the set $\mathbf{G}$ defined by $\sigma_i(\mathbf{h})=\mathbf{h}\circ (g_i)$. Let $i^{\m}$ be the index such that $g_ig_{i^{\m}}=e$. With this notation, we get the \textit{self-similar} representation
$$
(g_i)=((g_i),\ldots, (g_i),\underbrace{(e)}_{i^{\m}},(g_i),\ldots, (g_i))\sigma_i.
$$
For the group $\mathcal{G}(\partial \mathrsfs{\wt{C}}(G))$, we have
$$
(g_i)=((e),\ldots, (g_i),\underbrace{(e)}_{i^{\m}},(g_i),\ldots, (g_i))\sigma_i.
$$

\begin{proposition}
For any non trivial group $G$, the semigroups $\mathcal{S}(\partial\mathrsfs{C}(G))$ is free.
\end{proposition}
\begin{proof}
Suppose contrary to our claim that $\mathcal{S}(\partial \mathrsfs{C}(G))$ is not free. Hence there are $u=(u_1)\cdots (u_n)$ and $v=(v_1)\cdots (v_k)$, such that $u=v$ and $(u_i), (v_i)\in (G)$. Since the semigroup $\mathcal{S}(\partial \mathrsfs{C}(G))$ is cancellative, the last statement is equivalent to say that $uv^{-1}$ is a relation in $\mathcal{G}(\partial \mathrsfs{C}(G))$. We can suppose that among the relations of this form, $uv^{-1}$ is minimal with respect to the length. Since $(G)$ and $(G)^{-1}$ are invariant under the action of $\partial \mathrsfs{C}(G)$, we have from minimality that, for any $\textbf{g}\in \textbf{G}$ either $\textbf{g}\cdot uv^{-1}=uv^{-1}$ or $\overline{\epsilon_{(e)}(\textbf{g}\cdot uv^{-1})}=1$. Otherwise we would get a shorter relation of the same form. First observe that $u_1\cdots u_nv_k^{-1}\cdots v_1^{-1}=e$ (in $G$) and for every $i=1,\ldots,n$, $j=1,\ldots, k$ there exist elements $\textbf{g}_i, \textbf{h}_j$ such that $\textbf{g}_i\circ (u_1)\cdots (u_i)= \textbf{e}$ and $\textbf{h}_j\circ (u_1)\cdots (u_n)(v_k)^{-1}\cdots (v_i)^{-1}= \textbf{e}$. Let us consider the element $\textbf{g}_{n-1}$ and apply it in order to get a new word $|\overline{\epsilon_{(e)}(\textbf{g}_{n-1}\cdot uv^{-1})}|<|uv^{-1}|$, which must be reduced by free cancellation because of the minimality of $uv^{-1}$. On the other hand, by direct computation we have
$$
(\textbf{g}_{n-1}\circ (u_1)\cdots (u_{n-1}))\cdot (u_{n})=(u_n)
$$
and
$$
   (\textbf{g}_{n-1}\circ (u_1)\cdots (u_n))\cdot (v_k)^{-1}=(v_k)^{-1}
  $$
This means that there is no cancelation of occurrences $(x)(x)^{-1}$ but this is absurd. We have only to treat the case when $n=1$ (resp. $k=1$), in this case we may get a relation $\textbf{g}\cdot uv^{-1}$ with no occurrences in $(G)$ (resp. $(G)^{-1}$). By using a minimality argument similar to what seen before one gets the assertion.
\end{proof}

The previous result immediately gives a result on the growth of the generated group \cite{Harp}.

\begin{corollary}
The groups $\mathcal{G}(\partial\mathrsfs{C}(G))$ has exponential growth, for any non trivial group $G$.
\end{corollary}

We stress once more the fact that the elements $(g)$ and $(g^{\m})$ of $\mathcal{G}(\partial\mathrsfs{C}(G))$ are not inverses to each other. On the other hand we have the following result.

\begin{proposition}
For any non trivial group $G$, $\mathcal{G}(\partial\widetilde{ \mathrsfs{C}}(G))$ is not free. In particular, for all $g\in G$, one has $(g)^{-1}=(g^{\m})$.
\end{proposition}
\begin{proof}
Firstly let us order the elements of $G$ in such a way that $e$ is the first element, $g$ the second element and $g^{\m}$ the third one, under the condition that $g^2\neq e$. By using this convention, the self-similar representation of the generators $(g)$ and $(g^{\m})$ of $\mathcal{G}(\partial\widetilde{ \mathrsfs{C}}(G))$ is
$$
(g)=((e), (g), (e) , (g), \ldots, (g))\sigma
$$
and
$$
(g^{\m})=((e), (e), (g^{\m}) , \ldots, (g^{\m}))\sigma^{-1}
$$
for some permutation $\sigma\in Sym(|\mathbf{G}|)$. By using the fact that
$$
\mathbf{g}\longfr{(g^{\m})}{(e)}\mathbf{e} \textrm{ \,  and  \, } \mathbf{g^{\m}}\longfr{(g)}{(e)}\mathbf{e}
$$
we get
$$
(g)(g^{\m})=((e),(g)(g^{\m}), (e), (g)(g^{\m}), \ldots, (g)(g^{\m}))
$$
and, similarly
$$
(g^{\m})(g)=((e),(e), (g^{\m})(g), \ldots, (g^{\m})(g))
$$
It is an easy exercise to prove that $(g)(g^{\m})$ and $(g^{\m})(g)$ coincide with the trivial automorphism. If $g=g^{\m}$ the proof works in the same way by showing that $(g)^2$ is trivial.
\end{proof}

We end this section with the following easy observation

\begin{proposition}
The (finite) group $G$ is a quotient of $\mathcal{G}(\partial\mathrsfs{C}(G))$ and $\mathcal{G}(\partial\widetilde{ \mathrsfs{C}}(G))$.
\end{proposition}
\begin{proof}
It is enough to observe that if $u$ is a relation in $\mathcal{G}(\partial\mathrsfs{C}(G)$ or $\mathcal{G}(\partial\widetilde{ \mathrsfs{C}}(G))$, then $u$ is a loop in the Cayley graph of $G$ with set of generators $G$. This is equivalent to say that $u$ is trivial in $G$ (see Lemma \ref{lem:prefragilita}).
\end{proof}

As an example consider the 0-transition Cayley machine $\mathrsfs{C}(\mathbb{Z}_n)=(\mathbf{Z_{n}},(\mathbb{Z}_n),\circ,\cdot)$, with $n\in \mathbb{N}$ (see Fig. \ref{Fig: automata cayley 3}). As usual $\mathbb{Z}_n=\{0,1,\ldots, n-1\}$ with operation $i+j=i+j \textit{ mod }n$ and for $i\neq 0$ we denote by $-i$ the unique element $n-i$ in $\mathbb{Z}_n$. In this case we export the additive notation to the automaton, in order to get the following transitions
\begin{itemize}
\item $\mathbf{k}\vlongfr{(i)}{(i)}\mathbf{k+i}$ for all $k,i\in \mathbb{Z}_{n}$ such that $k\neq -i$;
\item $\mathbf{k}\vlongfr{(-k)}{(0)}\mathbf{0}$ for all $k\in \mathbb{Z}_{n}$.
\end{itemize}
In what follows we see which conditions should be satisfied by a minimal relation in this special case.
Given any word $u\in \widetilde{(\mathbb{Z}_n)}^{\ast}$, this is given by a sequence $u=(u_1)\cdots (u_m)$ where $(u_i)\in \widetilde{(\mathbb{Z}_n)}$ or, equivalently, by the word $u=(v_1)^{e_1}\cdots (v_m)^{e_m}$, where $(v_i)\in (\mathbb{Z}_n)$ and $e_i\in \{-1,+1\}$. By Lemmata \ref{lem:prefragilita} and \ref{lem:fragilita} $u$ is a relation of minimal length if and only if
\begin{enumerate}
  \item $\sum e_iv_i=0\textrm{ mod }n$
  \item for $\mathbf{h}\in\mathbf{Z_{n}}$ either $\mathbf{h}\cdot u=u$ or $\overline{\epsilon_{(e)}(\mathbf{h}\cdot u)}=1$.
\end{enumerate}

The first condition ensures that $u$ is a relation in $\mathbb{Z}_n$, the second one takes into account the fact that if we process $u$, and some letter $(u_i)$ is changed, then it is replaced by $(e)$ (or by $(e)^{-1}$). With any $u\in \widetilde{(\mathbb{Z}_n)}^{\ast}$ one can associate a (finite) set $S(u)$ of elements of $\mathbb{Z}_n$ defined by $S(u):=\{s_i(u): \ i=1,\ldots, m\}$, where $s_j(u)= \sum_{i=1}^{j} e_iv_i\textrm{ mod }n$. Given $u$, we consider the map $\phi_u:\{1,2,\ldots, |u|\}\rightarrow \mathbb{Z}_n$, such that $\phi_u(j)=s_j(u)$. Let $s_j(u)=q\in\mathbb{Z}_n$, and let $\textbf{p}$ the state of $\mathrsfs{C}(\mathbb{Z}_n)$ corresponding to the element $-q$, then for every $k\in \phi_u^{-1}(q)$ one gets one of the following four cases
\begin{itemize}
  \item if $e_k=e_{k+1}=1$, then
  $$
(\textbf{p}\circ (v_1)^{e_1}\cdots (v_{k-1})^{e_{k-1}})\cdot (v_k)=(e)
$$
$$
(\textbf{p}\circ (v_1)^{e_1}\cdots (v_k))\cdot (v_{k+1})=(v_{k+1})
$$
  \item if $e_k=1$ and $e_{k+1}=-1$, then
  $$
(\textbf{p}\circ (v_1)^{e_1}\cdots (v_{k-1})^{e_{k-1}})\cdot (v_k)=(e)
$$
$$
(\textbf{p}\circ (v_1)^{e_1}\cdots (v_k))\cdot (v_{k+1})^{-1}=(e)
$$
  \item if $e_k=-1$ and $e_{k+1}=1$, then
  $$
(\textbf{p}\circ (v_1)^{e_1}\cdots (v_{k-1})^{e_{k-1}})\cdot (v_k)^{-1}=(v_k)^{-1}
$$
$$
(\textbf{p}\circ (v_1)^{e_1}\cdots (v_{k})^{-1})\cdot (v_{k+1})=(v_{k+1})
$$
  \item if $e_k=e_{k+1}=-1$, then
$$
(\textbf{p}\circ (v_1)^{e_1}\cdots (v_{k-1})^{e_{k-1}})\cdot (v_k)^{-1}=(v_k)^{-1}
$$
$$
(\textbf{p}\circ (v_1)^{e_1}\cdots (v_{k})^{-1})\cdot (v_{k+1})^{-1}=(e)
$$
\end{itemize}
If $u$ is a relation of minimal length (then $\overline{\epsilon_{(e)}(\textbf{p}\cdot u)}$ is trivial), the cancellation of occurrences $(x)(x^{-1})$ is performed around at least one index $k\in \phi_u^{-1}(q)$, for some $q\in S(u)$. The only possibility can occur either for $e_k=1, e_{k+1}=-1$ or $e_k=-1, e_{k+1}=1$. Despite the easy combinatorial properties of the relations of minimal length, it seems to be hard to get an explicit description and a simply geometrical interpretation of such relations. From the computational point of view, by using the GAP package AutomGrp developed by Y. Muntyan and D. Savchuk \cite{MuSav14} we have not been able to find non-trivial relations for such examples of groups. This suggests that, if there are relations, they are very long with respect to the size of the generating set.

\section{Open Problems}
We give a list of open problems
\begin{prob}
 Is there a invertible transducer which is not bireversible (possibly with sink) defining a free non-abelian group acting transitively on the associated tree?
\end{prob}
\begin{prob}
Is there a characterization of strongly fragile words? Does the geometric argument described in Section \ref{sec: with sink} can be extended to give a full description of strongly fragile words?
\end{prob}
\begin{prob}
The groups generated by the dual of the 0-transition Cayley machines have exponential growth (see Section \ref{section: cayley}). What can be said about the amenability of such groups? More generally, is it possible to find a suitable output-coloring of such transducers in order to get free groups or free products of groups? This question can be specialized for the example of Cayley machines, where $G=\mathbb{Z}_n$. Are the groups generated by dual of 0-transition Cayley machine $\mathrsfs{C}(\mathbb{Z}_n)$ free? Are the groups generated by dual of 0-transition Cayley machine $\widetilde{\mathrsfs{C}}(\mathbb{Z}_n)$ free products?. In case, does there exists a simple combinatorial description of the relations?
\end{prob}
\begin{prob}
Find a (geometric) characterization of strongly fragile words. Does the shortest strongly fragile on $m$ occurrences have length $3(2^{m-1})-2$? This would imply that the shortest relation on $m$ occurrences of a group $\mathcal{G}(\mathrsfs{A})$ such that $\mathrsfs{S}_{Q}\hookrightarrow_{d} \mathrsfs{A}$, has length greater or equal to $3(2^{m-1})-2$.
\end{prob}
\begin{prob}
Before trying to attack the decidability/undecidability of the algorithmic problem \texttt{FINITE-PERIODIC-ORBITAL} it would be interesting to understand if the following simpler problem is decidable/undecidable:
\begin{itemize}
\item Input: An inverse transducer $\mathrsfs{A}=(Q,\wt{A},\cdot,\circ)$, an element $y\in \wt{Q}^{*}$ with $\oo{y}\neq 1$.
\item output: Is the Schreier graph centered in $y^{\omega}$ finite?
\end{itemize}
\end{prob}
\begin{prob}
Theorems \ref{theo: characterization free}, \ref{theo: finiteness schreier} and Corollary \ref{cor: trans implies positive free} support the idea that the more an automaton  (semi)group is transitive, the more the (semi)group generated by its dual is close to be free. For instance, if an automaton semigroup is transitive is it always true that the semigroup defined by its dual is free?
\end{prob}
\begin{prob} Is it true that no reversible invertible (or bireversible) transducer generates a Burnside group? If there would not be any bireversible transducers generating a Burnside group, then in view of Theorem \ref{theo: burnside} and Corollary \ref{cor: no trivial stabilizers RI not bireversible} examples of automata groups with all trivial stabilizers in the border should be sought outside the class of reversible invertible transducers.
\end{prob}

\bibliographystyle{plain}

\end{document}